\documentclass{SCIYA2017enOL}

    \newcommand{\be}{\begin{equation}}
    \newcommand{\ee}{\end{equation}}

    \newcommand{\nrm}[1]{\left\| #1 \right\|}

    \newcommand{\dx}{\mbox{d}\mathbf{x}}
    \newcommand{\dy}{\mbox{d}\mathbf{y}}
    \newcommand{\bx}{\mathbf{x}}
    \newcommand{\by}{\mathbf{y}}
    
    \def\g{\mbox{\boldmath $g$}}

    \newcommand\dt {{\Delta t}}
    \newcommand{\Cinv}{C_{\text{inv}}}
    \newcommand\Mh {\mathcal{M}_h}
    \renewcommand\i {\mathrm{i}}
     \def\0{\mbox{\boldmath $0$}}

    \def\g{\mbox{\boldmath $g$}}

\online
\begin{document}

\ensubject{fdsfd}

\ArticleType{ARTICLES}
\Year{2022}
\Month{January}%
\Vol{65}
\No{1}
\BeginPage{1} %
\DOI{10.1007/s11425-016-5135-4}
\ReceiveDate{January 1, 2022}
\AcceptDate{January 1, 2022}
\OnlineDate{January 1, 2022}

\title[]{Double stabilizations and convergence analysis of a second-order linear numerical scheme for the nonlocal Cahn--Hilliard equation}
{}

\author[1]{Xiao Li}{{xiao1li@polyu.edu.hk}}
\author[1,$\ast$]{Zhonghua Qiao}{{zqiao@polyu.edu.hk}}
\author[2]{Cheng Wang}{cwang1@umassd.edu}

\AuthorMark{Li X}

\AuthorCitation{Li X, Qiao Z H, Wang C}

\address[1]{Department of Applied Mathematics, The Hong Kong Polytechnic University, Hung Hom, Kowloon, Hong Kong}
\address[2]{Department of Mathematics, The University of Massachusetts, North Dartmouth, MA {\rm 02747}, USA}

\abstract{In this paper, we study a second-order accurate and linear numerical scheme for the nonlocal Cahn--Hilliard equation. The scheme is established by combining a modified Crank--Nicolson approximation and the Adams--Bashforth extrapolation for the temporal discretization, and by applying the Fourier spectral collocation to the spatial discretization. In addition, two stabilization terms in different forms are added for the sake of the numerical stability. We conduct a complete convergence analysis by using the higher-order consistency estimate for the numerical scheme, combined with the rough error estimate and the refined estimate. By regarding the numerical solution as a small perturbation of the exact solution, we are able to justify the discrete $\ell^\infty$ bound of the numerical solution, as a result of the rough error estimate. Subsequently, the refined error estimate is derived to obtain the optimal rate of convergence, following the established $\ell^\infty$ bound of the numerical solution. Moreover, the energy stability is also rigorously proved with respect to a modified energy. The proposed scheme can be viewed as the generalization of the second-order scheme presented in an earlier work, and the energy stability estimate has greatly improved the corresponding result therein.}



\maketitle


\section{Introduction}

In this paper, we study the nonlocal Cahn--Hilliard (NCH) equation~\cite{bates06a, bates09, bates05b, bates05a, bates06b, guan17a, guan14b, guan14a}
\begin{equation}
\label{equation-nCH}
\phi_t = \Delta ( \phi^3 - \phi + \varepsilon^2 {\cal L} \phi ),
\quad (\mathbf{x},t) \in \Omega \times (0,T],
\end{equation}
where $\Omega=\prod_{i=1}^d (-X_i,X_i)$ is a cuboid domain in $\mathbb{R}^d$ ($d=1,2,3$)
and $\phi=\phi(\mathbf{x},t)$ is the unknown function subject to periodic boundary condition on $\overline{\Omega}$. In the last term of the right-hand side, $\varepsilon>0$ is an interfacial parameter, and $\mathcal{L}$ is a nonlocal linear operator defined as
\begin{equation*}
\mathcal{L} \psi(\mathbf{x}) = \int_\Omega J(\mathbf{x}-\mathbf{y}) (\psi(\mathbf{x})-\psi(\mathbf{y})) \, \dy .
\end{equation*}
In more details, $J$ is a kernel function satisfying
\begin{itemize}
\item[(a)] $J(\bx)\ge 0$ for any $\bx\in\mathbb{R}^d$;
\item[(b)] $J$ is $\Omega$-periodic and even, that is, $J(-\bx)=J(\bx)$ for any $\bx\in\mathbb{R}^d$;
\item[(c)] $\frac{1}{2}\int_\Omega J(\bx)|\bx|^2 \,  \dx = 1$;
\item[(d)] $J$ is integrable on $\Omega$ and $\gamma_0:=\varepsilon^2 (J*1)-1>0$,
\end{itemize}
where $*$ stands for the periodic convolution~\cite{guan14a}
\[
(J * \psi) (\bx) = \int_\Omega J(\bx-\by) \psi(\by) \, \dy = \int_\Omega J(\by) \psi(\bx-\by) \, \dy.
\]
Using the condition (d), the nonlocal operator can also be rewritten as
\[
\mathcal{L} \psi = (J*1) \psi - J*\psi,
\]
and correspondingly, the NCH equation \eqref{equation-nCH} becomes
\begin{equation*}
\phi_t = \Delta ( \phi^3 + \gamma_0 \phi - \varepsilon^2 J * \phi  )
= \nabla \cdot ((3\phi^2+\gamma_0) \nabla \phi) - \varepsilon^2 \Delta J*\phi .
\end{equation*}
The positivity of $\gamma_0$ implies the diffusivity of the leading term $\nabla \cdot ((3\phi^2+\gamma_0) \nabla \phi)$, while the solution may perform some singular behavior without such a condition~\cite{bates05b,bates05a}.

Similar to the classic Cahn--Hilliard equation~\cite{cahn58}, the NCH equation \eqref{equation-nCH} can be viewed as the $H^{-1}$ gradient flow with respect to a free energy functional with nonlocal interaction effects. The energy functional reads as
\begin{equation}
\label{energy-nCH}
E (\phi) = \int_\Omega F(\phi) \, \dx + \frac{\varepsilon^2}{2} (\phi,\mathcal{L}\phi),
\end{equation}
where $F(\phi)=\frac{1}{4}(\phi^2-1)^2$ and $(\cdot,\cdot)$ denotes the standard $L^2$ inner product on $\Omega$. Due to energetic variational structure, the solution to the NCH equation decreases the energy \eqref{energy-nCH} in time, i.e., $\frac{d}{dt} E(\phi(t)) \le 0$. In addition, as a common property of $H^{-1}$ gradient flows, the mass conservation is obvious in the sense that $\frac{d}{dt}\int_\Omega \phi(\bx,t) \, \dx = 0$.

The NCH equation \eqref{equation-nCH} has attracted increasingly attention and been applied to a variety of areas, including material sciences, image processing, finance, etc. In material sciences, the NCH equation and a few other related formulations arise as the mesoscopic model of interacting particle systems and phase transitions~\cite{fife03,hornthrop01}. In the dynamic density functional theory~\cite{archer04a,archer04b}, the solution describes the mesoscopic particle density and the interaction kernel is the two-particle direct correlation function. In comparison with the classic Cahn--Hilliard equation, the NCH equation performs more flexibility to describe more types of physical processes and phenomena by appropriately choosing interaction kernel functions. At the theoretical level, the well-posedness of the NCH equation with an integrable kernel function and the Neumann or Dirichlet boundary condition was studied by Bates and Han~\cite{bates05b,bates05a}, and it was claimed in~\cite{guan14a} that the existence and uniqueness of the periodic solution to the NCH equation may be established by using similar techniques. We refer the readers to~\cite{du12a, fife03} for some reviews of nonlocal diffusion models and parabolic-like evolution equations. We also refer the readers to~\cite{AinsworthMao17,SongXuKa16,TangYuZh19} for some other different forms of the nonlocal Cahn--Hilliard equations. At the numerical level, some researches have been devoted to designing efficient algorithms for nonlocal diffusion equations~\cite{du12a}, the nonlocal Allen--Cahn equation (the $L^2$ gradient flow with respect to the energy \eqref{energy-nCH})~\cite{du19,du16}, and some other nonlocal models~\cite{bates06b}. For the NCH equation, one of the main difficulties comes from the existence of both the nonlocal term and the Laplacian of nonlinear terms. Due to the energetic variational structure of the model, the numerical algorithms inheriting the energy dissipation law are always highly desired. To this end, the nonlocal term and the nonlinear term need to be addressed carefully. Guan et al~\cite{guan17a,guan14b,guan14a} developed first- and second-order convex splitting schemes for the NCH equation and proved the energy stability and convergence. In particular, the nonlinear term was treated implicitly to guarantee the energy stability, under the framework of the convex splitting approach (see also~\cite{eyre98, LiQiZh16, QiaoZhTa11, shen12, wang10, wise09}). As a result, an iteration solver becomes inevitable in the numerical implementation, which comes from the nonlinearity of the schemes. In addition, the nonlocal term was set into the explicit part to contribute only the right-hand side of the nonlinear system, so that multiple evaluations could be avoided in the nonlinear iteration at each time step.

To further simplify the computation efforts, some linear numerical schemes have been developed for the NCH equation~\cite{du18, LiX21b, LiX21a}, by applying the stabilization technique~\cite{shen10, xu06} to preserve the energy stability. The first-order scheme~\cite{du18} followed the idea of the standard stabilized implicit-explicit method and a theoretical justification of the energy stability and convergence analysis was presented in~\cite{LiX21a}. Moreover, the second-order backward differentiation formula (BDF2) was applied to construct a second-order accurate stabilized linear scheme~\cite{LiX21b} with the explicit extrapolation adopted for the nonlinear term and concave expansive term. This BDF2 scheme was proved to be energy stable with respect to a modified energy, which is an $O (\dt)$ approximation of the original energy \eqref{energy-nCH} at the numerical level. The convergence analysis was also carried out via the induction argument.
We refer the readers to \cite{LiaoSoTaZh21,LiaoZh20} and the references therein for more applications of the BDF2 method and \cite{JuLiQi22,LiShen22,ShenXuYa19,YangZh20} for more linear schemes for some other gradient flow equations.

Other than the BDF2 approach, another second-order stabilized linear scheme, based on the modified Crank--Nicolson discretization, has been studied in the existing work~\cite{du18}. This modified Crank--Nicolson scheme takes the form of
\begin{align}
\frac{\phi^{n+1} - \phi^n}{\dt}
   & = \Delta_N \Bigl(   \frac32 (\phi^n)^3 - \frac12 (\phi^{n-1})^3
   - \Big(\frac32 \phi^n - \frac12 \phi^{n-1}\Big) \nonumber \\
& \quad
   + \varepsilon^2 {\cal L}_N \Big( \frac34 \phi^{n+1} + \frac14 \phi^{n-1} \Big)
   + A_0 ( \phi^{n+1} - 2 \phi^n + \phi^{n-1}) \Bigr) . \label{scheme-2nd-0}
\end{align}
A modified energy inequality has been established in~\cite{du18} as
\begin{equation}
\label{scheme-2nd-0-energy}
\tilde{E}_N (\phi^{n+1}, \phi^n) \le \tilde{E}_N (\phi^n, \phi^{n-1}) + \frac{4A_0}{3} \| \phi^{n+1} - \phi^n \|_2^2 ,
\end{equation}
if the stabilization constant $A_0$ satisfies
\begin{equation}
\label{condition-A-0}
A_0 \ge  \max \Big\{  \frac43 ( \| \phi^n \|_\infty^2 + \| \phi^{n-1} \|_\infty^2 ) - \frac83 ,
\frac23 ( \| \phi^{n+1} \|_\infty^2 + 2 \| \phi^n \|_\infty^2 ) \Big\} .
\end{equation}
The operators with subindex $N$, as well as the discrete norms, represent the corresponding spatially-discrete versions; the precise definitions will be given in the next section; the term $\tilde{E}_N (\phi^n, \phi^{n-1})$ is a modified energy defined by the original energy $E_N (\phi^n)$,  with a perturbation of order $O(\dt^2)$. However, we notice that the inequality \eqref{scheme-2nd-0-energy} is not a rigorous energy stability estimate, since it does not ensure a global-in-time bound of the energy functional due to the lack of a theoretical control of the increment term $\| \phi^{n+1} - \phi^n \|_2^2$, although it is formally expected to be of order $O(\dt^2)$. In addition, the $\ell^\infty$ norms of the numerical solutions at time steps $t_{n-1}$, $t_n$ and $t_{n+1}$ are involved on the right-hand side of \eqref{condition-A-0}. As a result, such a lower bound for the constant $A_0$ has not been justified at a theoretical level in~\cite{du18}.

The primary goal of this work is to present a complete analysis of the energy stability and convergence for the second-order stabilized linear scheme \eqref{scheme-2nd-0}.
In particular, we have to slightly modify the scheme \eqref{scheme-2nd-0} to ensure the theoretical properties.
In more details, an additional $O(\dt^2)$ stabilization term, in the form of $A_1 \dt \Delta_N (\phi^{n+1} - \phi^n)$ (a Douglas--Dupont regularization term), is added to the right-hand side. As a result, double stabilization terms are involved in the numerical scheme to facilitate the theoretical analysis. The double stabilization technique has been used to analyze classic Allen--Cahn and Cahn--Hilliard equations~\cite{WangL2018, WangL2018b, WangL2019, WangL2020}, where the lower bounds of the constants $A_0$ and $A_1$ depend on the $\ell^\infty$ bound of the unknown numerical solutions, which have not been theoretically determined. To justify the lower bounds of the constants $A_0$ and $A_1$, a direct analysis provided in~\cite{LiD2017, LiD2017b, LiD2016a} for the classic Cahn--Hilliard equation may hardly be extended to this numerical scheme, due to the lack of higher-order diffusion terms. Instead, we view the numerical solution as a perturbation of the exact solution to \eqref{equation-nCH}, and use the convergence estimate to obtain an $\ell^\infty$ bound of the numerical solution. In more details, a high-order consistency analysis is performed, so that the uniform $\ell^\infty$ bound of the numerical solution, as well as its discrete temporal derivative, can be theoretically justified. Moreover, one crucial difference with the standard error estimate is associated with the fact that, we have to adopt $(-\Delta_N)^{-1}(\hat{e}^{n+1}-\hat{e}^n)$ to test the error equation with respect to the numerical error function $\hat{e}^n$, instead of testing $(-\Delta_N)^{-1}\hat{e}^{n+1}$ as in the standard error estimate (where $(-\Delta_N)^{-1}$ is a spatial discrete operator to be defined in the next section). Therefore, the key point of the convergence analysis is to use the discrete temporal derivative of the error function as the test function, rather than the error function directly, which would provide a higher-order temporal truncation error to match the modified Crank--Nicolson discretization for the temporal derivative. As a result of the convergence estimate, we obtain a uniform $\ell^\infty$ bound of the numerical solution. This in turn recovers the \emph{a priori} assumption, and the lower bounds for both $A_0$ and $A_1$ become available at a theoretical level.

Although a BDF2 scheme has been recently investigated for the NCH equation in~\cite{LiX21b}, the numerical scheme proposed in this paper (the scheme \eqref{scheme-2nd-1} given later) still performs significantly in some aspects. First, the constraints of the stabilizing constants $A_0$ and $A_1$ for the energy stability are of order $O (M_0^2)$ (with $M_0$ the supremum norm of the exact solution, as well as its temporal derivative), in comparison with the order $O (M_0^4)$ for the BDF2 scheme.
In other words, the lower bounds required for $A_0$ and $A_1$ are expected to be smaller for the Crank--Nicolson scheme at a theoretical level. Second, as mentioned above, the modified energy defined for the BDF2 scheme possesses a deviation of order $O(\dt)$ away from the original energy functional. For the proposed Crank--Nicolson scheme, we will prove the energy stability with respect to a modified energy with a deviation of order $O(\dt^2)$ away from the original energy  functional. This fact implies that the modified energy dissipation law becomes closer to the original physical system, in comparison with the BDF2 scheme reported in~\cite{LiX21b}.

The rest of this paper is organized as follows. The second-order stabilized linear numerical scheme, obtained by modifying the existing algorithm \eqref{scheme-2nd-0}, is presented in the fully-discrete version in Section~\ref{sect_scheme}. Some spatial discretization notations are introduced.
In Section~\ref{sect_convergence}, we conduct the convergence analysis for the proposed scheme by the induction argument, including the higher-order consistency estimate, a rough error estimate, and a refined error estimate. In addition, the infinity-norm of the numerical solution is justified as a by-product of the convergence result. Subsequently, the energy stability of the proposed scheme is proved in Section~\ref{sect_stability}.
Some numerical experiments are conducted in Section~\ref{sect_numerical}
to verify the second-order temporal convergence rates and the energy dissipation property.
Finally, some concluding remarks are given in Section~\ref{sect_conclusion}.

\section{Second-order stabilized linear numerical scheme}
\label{sect_scheme}

In this section, we develop the fully-discrete second-order scheme for the NCH equation \eqref{equation-nCH}.
First, we summarize some notations for the 2-D Fourier spectral collocation method for the spatial discretization.
An extension to the 3-D case is straightforward.

For simplicity of notations, we consider the square domain $\Omega=(-X,X)^2$. For any given even number $N$, let $h=2X/N$ be the size of the uniform mesh, denoted by $\Omega_h$, composed of the nodes $(x_i,y_j)$ with $x_i=-X+ih$ and $y_j=-X+jh$ for $1\le i,j\le N$. The space of all $\Omega_h$-periodic grid functions is defined as
\[
\Mh = \{f:\mathbb{Z}^2\to\mathbb{R}\,|\,f_{i+pN,j+qN}=f_{ij} \text{ for $1\le i,j\le N$ and $p,q\in\mathbb{Z}$}\}.
\]
For any grid functions $f,g\in\Mh$,
the $\ell^2$ inner product, the $\ell^p$ norm ($1\le p<\infty$), and the $\ell^\infty$ norm
are defined respectively as
\[
\langle f,g \rangle = h^2 \sum_{i,j=1}^N f_{ij}g_{ij}, \quad
\|f\|_p = \langle |f|^p,1 \rangle^{\frac{1}{p}}, \quad
\|f\|_\infty = \max_{1\le i,j\le N} |f_{ij}|.
\]
In particular, the $\ell^2$ norm can also be expressed as $\|f\|_2=\sqrt{\langle f,f \rangle}$. A subspace of $\Mh$ collecting all grid functions with zero mean is denoted by
$\Mh^0 = \{f\in\Mh \,|\, \langle f,1 \rangle = 0\}$.

For $f\in\Mh$, we have the discrete Fourier expansion
\[
f_{ij} = \sum_{k,l=-N/2+1}^{N/2} \hat{f}_{kl} \exp\Big(\frac{\i\pi}{X} (kx_i+ly_j)\Big), \quad
\hat{f}_{kl} = \frac{1}{N^2} \sum_{i,j=1}^N f_{ij} \exp\Big(-\frac{\i\pi}{X} (kx_i+ly_j)\Big).
\]
The Fourier pseudo-spectral approximations to the first and second partial derivatives in the $x$-direction are defined as
\begin{align*}
D_x f_{ij} & = \sum_{k,l=-N/2+1}^{N/2} \frac{\i k\pi}{X}\hat{f}_{kl} \exp\Big(\frac{\i\pi}{X} (kx_i+ly_j)\Big), \\
D_x^2 f_{ij} & = \sum_{k,l=-N/2+1}^{N/2} \Big(-\frac{(k\pi)^2}{X^2}\Big)\hat{f}_{kl}
\exp\Big(\frac{\i\pi}{X} (kx_i+ly_j)\Big).
\end{align*}
The operators $D_y$ and $D_y^2$ in the $y$-direction can be defined in the similar way. For any $f\in\Mh$ and $\mathbf{f}=(f^1,f^2)^T\in\Mh\times\Mh$, the discrete gradient, divergence, and Laplace operators are defined respectively as
\[
\nabla_N f = \binom{D_x f}{D_y f}, \quad
\nabla_N \cdot \mathbf{f} = D_x f^1 + D_y f^2, \quad
\Delta_N f = D_x^2 f + D_y^2 f.
\]
For any $f,g\in\Mh$ and $\mathbf{g}\in\Mh\times\Mh$, we have the following summation-by-parts formulas~\cite{gottlieb12a, gottlieb12b, LiX21a}:
\[
\langle f, \nabla_N \cdot \mathbf{g} \rangle = - \langle \nabla_N f, \mathbf{g} \rangle, \quad
\langle f, \Delta_N g \rangle = - \langle \nabla_N f, \nabla_N g \rangle = \langle \Delta_N f, g \rangle.
\]
In addition, $-\Delta_N$ is self-adjoint and positive definite on $\Mh^0$, and thus $(-\Delta_N)^{-1}$ exists and also positive definite on $\Mh^0$. Moreover, for any $f,g\in\Mh^0$, we define the discrete $H^{-1}$ inner product and the discrete $H^{-1}$ norm as
\begin{align*}
\langle f, g \rangle_{-1,N} & = \langle f, (-\Delta_N)^{-1} g \rangle
= \langle (-\Delta_N)^{-\frac{1}{2}} f, (-\Delta_N)^{-\frac{1}{2}} g \rangle, \\
\|f\|_{-1,N} & = \sqrt{\langle f, f \rangle_{-1,N}} = \|(-\Delta_N)^{-\frac{1}{2}} f\|_2.
\end{align*}

To define the discrete version of the nonlocal operator $\mathcal{L}$, we need the discrete convolution notation. The following definition follows the similar notations in~\cite{guan14a,LiX21a}.
For any $f,\phi\in\Mh$, the discrete convolution $f * \phi \in \Mh$ is defined at a componentwise level:
\[
(f * \phi)_{ij} = h^2 \sum_{p,q=1}^N f_{i-p,j-q} \phi_{pq}, \quad 1 \le i,j \le N.
\]
In a recent work~\cite{LiX21a}, the following preliminary estimate has been established for the discrete convolution, which will be used in the later analysis.

\begin{lemma} [See~\cite{LiX21a}] \label{lem:1}
Suppose $\phi, \psi$ are two periodic grid functions.  Assume that $\mathsf{f} \in C_{\rm per}^1(\Omega)$ is even and define its grid restriction via $f_{ij} := \mathsf{f}(x_i,y_j)$. Then for any $\alpha > 0$, we have
	\begin{equation}
   \left| \langle f * \phi , \Delta_N \psi \rangle \right|  \le \alpha \| \phi \|^2_2 + \frac{C_\mathsf{f}}{\alpha} \| \nabla_N \psi \|^2_2  ,  \label{lem 1:0}
	\end{equation}
where $C_\mathsf{f}$ is a positive constant that depends on $\mathsf{f}$ but is independent of $h$.
\end{lemma}

Given a kernel function $J$ satisfying the conditions (a)--(d), the discrete version of the nonlocal operator can be defined as
\[
\mathcal{L}_N \phi = (J * 1) \phi - J * \phi, \quad \forall \, \phi \in \Mh.
\]

Finally, we present the second-order stabilized linear numerical scheme studied in this paper. Given $\dt$ a uniform time step size, we set $\{t_k=k\dt\}$ as the nodes in the time interval and denote by $\phi^k$ the numerical solution at time $t=t_k$. The fully-discrete scheme is proposed as follows: given $\phi^n,\phi^{n-1}\in\Mh^0$ ($n\ge1$), find $\phi^{n+1}\in\Mh^0$ such that
\begin{align}
\frac{\phi^{n+1} - \phi^n}{\dt}
   & = \Delta_N \Bigl(   \frac32 (\phi^n)^3 - \frac12 (\phi^{n-1})^3  - \breve{\phi}^{n+1/2}
   + A_0 ( \phi^{n+1} - 2 \phi^n + \phi^{n-1}) \nonumber
\\
  & \quad
   + A_1 \dt ( \phi^{n+1} - \phi^n )
   + \varepsilon^2 {\cal L}_N \Big( \frac34 \phi^{n+1} + \frac14 \phi^{n-1} \Big) \Bigr) , \label{scheme-2nd-1}
\end{align}
with $\breve{\phi}^{n+1/2} = \frac32 \phi^n - \frac12 \phi^{n-1}$. The case of $A_1 =0$ yields the algorithm~\eqref{scheme-2nd-0} studied in~\cite{du18}. In addition to $A_0 ( \phi^{n+1} - 2 \phi^n + \phi^{n-1})$, the term $A_1 \dt ( \phi^{n+1} - \phi^n )$ is another stabilization term, which stands for the Douglas--Dupont regularization. Therefore, double stabilizations have been involved in the proposed scheme. The later analysis will reveal that the stabilization term $A_1 \dt ( \phi^{n+1} - \phi^n )$ does not contribute to the convergence estimate, while it is crucial to the energy stability estimate.

In addition, since the proposed scheme~\eqref{scheme-2nd-1} is a two-step algorithm, we have to give some remarks on the initialization process to obtain the numerical solution $\phi^1$. A simple choice of single-step algorithms to generate $\phi^1$ is the first-order stabilized linear scheme proposed and studied in~\cite{du18,LiX21a}, in which a second-order temporal accuracy could be obtained in the first step (see~\cite{guo16, guo2021} for the related analysis for the classic Cahn--Hilliard equation). However, for the proposed scheme~\eqref{scheme-2nd-1}, a higher-order approximation at time $t=t_1$ is needed in the theoretical analysis. Therefore, a second-order accurate numerical algorithm is highly preferred in the first time step. For instance, the discrete gradient scheme~\cite{du91, mclachlan99} turns out to be a one-step second-order accurate and energy stable scheme, so it gives a third-order approximation at time $t=t_1$ if the exact initial data is imposed for $\phi^0$. While the discrete equations are inevitably nonlinear in this approach, the explicit second-order Runge--Kutta method can be another choice, with the desired accuracy but sacrificing the energy dissipation property.

\section{Convergence analysis}
\label{sect_convergence}

Denote by $\Phi$ the exact solution to \eqref{equation-nCH}.
The existence and uniqueness of a smooth periodic solution to the NCH equation \eqref{equation-nCH} with smooth periodic initial data may be established using techniques developed by Bates and Han in~\cite{bates05b, bates05a}, from which one can obtain
\begin{equation}
\label{regularity}
\nrm{\Phi}_{L^\infty(0,T;L^\infty)} 
+ \nrm{ \Phi_t}_{L^\infty(0,T;L^\infty)} \le C  ,
\end{equation}
for any $T >0$.

Define $\Phi_N (\, \cdot \, ,t) := {\cal P}_N \Phi (\, \cdot \, ,t)$, the (spatial) Fourier projection of the exact solution into ${\cal B}^K$, the space of trigonometric polynomials of degree up to and including $K:=N/2$.  The following projection approximation is standard: if $\Phi \in L^\infty(0,T;H^\ell_{\rm per}(\Omega))$ for some $\ell\in\mathbb{N}$, then
\begin{equation}
\| \Phi_N - \Phi \|_{L^\infty(0,T;H^m)}
   \le C h^{\ell-m} \| \Phi \|_{L^\infty(0,T;H^\ell)},  \quad \forall \ 0 \le m \le \ell .
	\label{projection-est-0}
\end{equation}
We denote $\Phi_N^k=\Phi_N(\, \cdot \, , t_k)$ and $\Phi^k=\Phi (\, \cdot \, , t_k)$ with $t_k = k\dt$, and $\phi_N^k := {\mathcal P}_h \Phi_N (\, \cdot \, , t_k)$ the values of $\Phi_N$ at discrete grid points at time $t_k$. Since $\Phi_N \in {\cal B}^K$ and $1 \in {\cal B}^K$, we have the mass conservative property at the discrete level, i.e.,
\begin{align*}
\overline{\phi_N^k} & = \frac{1}{|\Omega|}\int_\Omega \, \Phi_N ( \cdot, t_k) \, \dx
= \frac{1}{|\Omega|}\int_\Omega \, \Phi ( \cdot, t_k) \, \dx \nonumber \\
& = \frac{1}{|\Omega|}\int_\Omega \, \Phi ( \cdot, t_{k-1}) \, \dx
= \frac{1}{|\Omega|}\int_\Omega \, \Phi_N ( \cdot, t_{k-1}) \, \dx = \overline{\phi_N^{k-1}},
\quad \forall \, k \in \mathbb{N} .
\label{mass conserv-1}
\end{align*}
We use the mass conservative projection for the initial data:  $\phi^0 = {\mathcal P}_h \Phi_N (\, \cdot \, , t=0)$, that is, $\phi^0_{ij} := \Phi_N (x_i, y_j, t=0)$. Thus, the solution to the numerical scheme \eqref{scheme-2nd-1} is also mass conservative at the discrete level:
\begin{equation*}
\overline{\phi^k} = \overline{\phi^{k-1}} ,  \quad \forall \, k \in \mathbb{N} .
\label{mass conserv-2}
\end{equation*}
Of course, based on the regularity assumption~\eqref{regularity}, we have
\begin{equation*}
 \max_{1\le k \le N_t} \|\phi_N^k\|_\infty
  + \max_{1\le k \le N_t} \bigg\| \frac{\phi_N^k - \phi_N^{k-1}}{\dt} \bigg\|_\infty
  < C^* ,
\label{IPDE-corrected-solution-stabilities}
\end{equation*}
where $N_t:=\lfloor T/\dt \rfloor$ for any given $T>0$.

Because of the fact that $\phi_N^k$ and $\Phi_N^k$ are identical on the discrete grid points, we just use the notation $\Phi_N^k$ in the following discussions for simplicity of presentation. With initial data of sufficient regularity, we can assume that the exact solution has regularity as
		\begin{equation*}
			\Phi \in \mathcal{R} := H^4 (0,T; C_{\rm per}^0(\overline{\Omega})) \cap H^3 (0,T; C_{\rm per}^2(\overline{\Omega}))  \cap L^\infty (0,T; C_{\rm per}^{m+2}(\overline{\Omega}))  ,  \quad m \ge 3 . 			
			\label{assumption:regularity.1}
		\end{equation*}
		
\begin{theorem}
\label{thm:convergence}
Given $T>0$, suppose the periodic solution to the NCH equation \eqref{equation-nCH}, given by $\Phi(x,y,t)$ on $\Omega$ for $0<t\le T$, is sufficiently smooth.
Meanwhile, the following assumption is made for the constants $A_0$ and $A_1$:
\begin{equation}
  A_0 \ge \frac{3 M_0^2}{2}  ,  \quad \mbox{with} \, \, \,
  M_0 =  1 + C^*, \, \, \, C^* = \max_{1\le k \le N_t} ( \| \Phi_N^k \|_\infty
  + \| \partial_t \Phi_N^k \|_\infty ) , \quad A_1 \ge 0 . \label{condition-A-1}
\end{equation}
Then, provided that $\dt$ and $h$ are sufficiently small, under linear refinement path constraint $C_1 h \le \dt \le C_2 h$ with $C_1$ and $C_2$ any fixed constants, we have the following error estimate
\begin{equation}
\label{convergence-0}
\| \Phi_N^n - \phi^n \|_2 \le C (\dt^2 + h^m) ,
\end{equation}
for all positive integers $n$ such that $n\dt \le T$, where $C>0$ is independent of $h$ and $\dt$.
\end{theorem}

The key point in the convergence proof is that, a higher-order consistency analysis is necessary to provide a higher-order truncation error, so that the desired $\ell^\infty$ bound of the numerical error can be recovered with the help of the inverse inequality. In fact, this approach has been adopted for the numerical analysis of a large family of nonlinear PDEs, see, e.g., \cite{baskaran13b, duan21a, duan20a, E95, guan17a, guan14a, LiuC2021, STWW03, WLJ04, WangL15}. With the higher-order truncation error established for the constructed approximation solution, we perform the stability estimates for the numerical error function. Meanwhile, it turns out to be impossible to obtain the expected results directly, due to the complicated nonlinear expansion. We have to divide this part into two steps. First, a rough estimate is performed to obtain the $\ell^\infty$ bound of the numerical solution, as well as its temporal derivative. Subsequently, a refined estimate is carried out to derive the desired result of convergence rate, based on the $\ell^\infty$ bound obtained by the rough estimate. In particular, instead of testing the error equation by $(-\Delta_N)^{-1}\hat{e}^{n+1}$, we adopt a test function in the form of $(-\Delta_N)^{-1}(\hat{e}^{n+1}-\hat{e}^n)$.

\subsection{Higher-order consistency analysis; asymptotic expansion}

With the Taylor expansion in time and the approximation estimate~\eqref{projection-est-0},
we know that the Fourier projection solution $\Phi_N$ solves the discrete equation
\begin{align*}
\frac{\Phi_N^{n+1} - \Phi_N^n}{\dt}
& = \Delta_N \Bigl( \frac32 (\Phi_N^n)^3 - \frac12 (\Phi_N^{n-1})^3 - \breve{\Phi}_N^{n+1/2}
+ A_0 (\Phi_N^{n+1} - 2 \Phi_N^n + \Phi_N^{n-1}) \nonumber
\\
  & \quad
   + A_1 \dt ( \Phi_N^{n+1} - \Phi_N^n )
   + \varepsilon^2 {\cal L}_N \Big( \frac34 \Phi_N^{n+1} + \frac14 \Phi_N^{n-1} \Big) \Bigr)
   + \tau_0^{n+1}  ,\label{truncation-error}
\end{align*}
where $\breve{\Phi}_N^{n+1/2} = \frac32 \Phi_N^n - \frac12 \Phi_N^{n-1}$
and $\tau_0^{n+1}$ is the truncation error determined by
\begin{align*}
\tau_0^{n+1} & = \bigg( \frac{\Phi_N^{n+1} - \Phi_N^n}{\dt} - \partial_t \Phi_N(t_{n+\frac12}) \bigg)
+ \varepsilon^2 \Delta \mathcal{L} \Big( \Phi_N(t_{n+\frac12}) - \frac34 \Phi_N^{n+1} - \frac14 \Phi_N^{n-1}\Big) \\
& \quad + \Delta \Big( \Phi_N(t_{n+\frac12})^3 - \frac32 (\Phi_N^n)^3 + \frac12 (\Phi_N^{n-1})^3\Big)
- \Delta \Big(\Phi_N(t_{n+\frac12}) - \frac32 \Phi_N^n + \frac12 \Phi_N^{n-1} \Big) \\
& \quad - \Delta_N \Bigl( A_0 (\Phi_N^{n+1} - 2 \Phi_N^n + \Phi_N^{n-1}) + A_1 \dt ( \Phi_N^{n+1} - \Phi_N^n )\Bigr) \\
& \quad 
+ \varepsilon^2 (\Delta \mathcal{L} - \Delta_N  {\cal L}_N) \Bigl( \frac34 \Phi_N^{n+1} + \frac14 \Phi_N^{n-1} \Bigr) 
+ \Delta \Big( \mathcal{P}_N(\Phi(t_{n+\frac12})^3) - \Phi_N(t_{n+\frac12})^3 \Big)\\
& \quad + (\Delta - \Delta_N) \Bigl( \frac32 (\Phi_N^n)^3 - \frac12 (\Phi_N^{n-1})^3 - \frac32 \Phi_N^n + \frac12 \Phi_N^{n-1}\Bigr).
\end{align*}
Note that we have assumed $\Phi \in \mathcal{R}$.
By the Taylor expansion with the integral remainder, one can easily conclude that the summation of the first three lines of the right-hand side of $\tau_0^{n+1}$ is bounded by $C \dt^2$;
by the Fourier spectral approximation, the rest terms has the bound $C h^m$.
In summary, we have $\| \tau_0^{n+1} \|_{-1,N} \le C (\dt^2 + h^m)$.
However, this local truncation error will not be enough to recover the $\ell^\infty$ bound of the numerical solution and its discrete temporal derivative, due to the second-order accuracy in time. To remedy this, we construct a supplementary field $\Phi^{(2)}_{\dt}$  and introduce the approximate solution
	\begin{equation}
\hat{\Phi} = \Phi_N + \dt^2 {\cal P}_N \Phi^{(2)}_{\dt}  .
	\label{consistency-1}
	\end{equation}
As a result of this construction, a higher $O (\dt^3 + h^m)$ consistency is satisfied with the given numerical scheme~\eqref{scheme-2nd-1}.  The constructed field $\Phi^{(2)}_{\dt}$ will be obtained using a perturbation expansion and depends only on the exact solution $\Phi$.

An application of the temporal discretization in the numerical scheme~\eqref{scheme-2nd-1} to the Fourier projection solution $\Phi_N$ indicates that
\begin{align}
\frac{\Phi_N^{n+1} - \Phi_N^n}{\dt}
& =  \Delta \Bigl( \frac32 (\Phi_N^n)^3 - \frac12 (\Phi_N^{n-1})^3 - \breve{\Phi}_N^{n+1/2}
+ A_0 (\Phi_N^{n+1} - 2 \Phi_N^n + \Phi_N^{n-1}) \nonumber
\\
  & \quad
   + A_1 \dt ( \Phi_N^{n+1} - \Phi_N^n )
   + \varepsilon^2 {\cal L} \Big( \frac34 \Phi_N^{n+1} + \frac14 \Phi_N^{n-1} \Big)  \Bigr)
   + \dt^2 \g^{(2)}(\cdot,t_{n+1/2}) + O (\dt^3)   , \label{consistency-2-1}
\end{align}
which comes from the Taylor expansion in time. In fact, the function $\g^{(2)}(\bx,t)$ is smooth enough and depends only on the higher-order derivatives of $\Phi_N$. In turn, the temporal correction function $\Phi^{(2)}_{\dt}$ is given by the solution of the following  linear differential equation
	\begin{eqnarray}
\partial_t \Phi^{(2)}_{\dt}  = \Delta  \Bigl(  3 (\Phi_N)^2 \Phi^{(2)}_{\dt} - \Phi^{(2)}_{\dt}
 + \varepsilon^2 {\cal L} \Phi^{(2)}_{\dt} \Bigr) - \g^{(2)}  .
	\label{consistency-2-2}
	\end{eqnarray}
In fact, the existence and uniqueness of the solution to \eqref{consistency-2-2} follows the standard argument for parabolic equations~\cite{Temam2001}, and this solution depends only on the profile $\Phi_N$ and is smooth enough. Similar to \eqref{consistency-2-1}, an application of the temporal discretization to $\Phi^{(2)}_{\dt}$ implies that
\begin{align}
  &
 \frac{ (\Phi^{(2)}_{\dt})^{n+1} - ( \Phi^{(2)}_{\dt} )^n }{\dt} \nonumber
\\
 & =
\Delta \Bigl( 3 \Big( \frac32 (\Phi_N^n)^2 (\Phi^{(2)}_{\dt})^n -  \frac12 (\Phi_N^{n-1})^2 (\Phi^{(2)}_{\dt})^{n-1} \Big)
  - (\breve{\Phi}^{(2)}_{\dt})^{n+1/2}  + \varepsilon^2 {\cal L} \Big( \frac34 ( \Phi^{(2)}_{\dt} )^{n+1}
   + \frac14 ( \Phi^{(2)}_{\dt} )^{n-1}\Big) \nonumber
\\
  & \quad
   + A_0 ( (\Phi^{(2)}_{\dt})^{n+1} - 2 (\Phi^{(2)}_{\dt})^n + (\Phi^{(2)}_{\dt})^{n-1} )
   + A_1 \dt ( (\Phi^{(2)}_{\dt})^{n+1} - (\Phi^{(2)}_{\dt})^n ) \Bigr)
    - \g^{(2)}(\cdot,t_{n+1/2}) + O (\dt^2) ,\label{consistency-2-3} 	
\end{align}
with $(\breve{\Phi}^{(2)}_{\dt})^{n+1/2} = \frac32 ( \Phi^{(2)}_{\dt})^{n} - \frac12 (\Phi^{(2)}_{\dt})^{n-1}$.  A combination of \eqref{consistency-2-1} and \eqref{consistency-2-3} results in the following higher-order consistency estimate:
\begin{align*} 	
 \frac{ \hat{\Phi}^{n+1} - \hat{\Phi}^n }{\dt} & =
\Delta \Bigl( \frac32 ( \hat{\Phi}^n  )^3 - \frac12 ( \hat{\Phi}^{n-1} )^3 - \breve{\hat{\Phi}}^{n+1/2}
+ A_0 ( \hat{\Phi}^{n+1} - 2 \hat{\Phi}^n + \hat{\Phi}^{n-1} ) \nonumber \\
  & \quad
   + A_1 \dt ( \hat{\Phi}^{n+1} - \hat{\Phi}^n )
   + \varepsilon^2 {\cal L} \Big( \frac34 \hat{\Phi}^{n+1}  + \frac14 \hat{\Phi}^{n-1} \Big) \Bigr)
    + O (\dt^3) ,
\end{align*}
with $\breve{\hat{\Phi}}^{n+1/2}  = \frac32 \hat{\Phi}^n -\frac12  \hat{\Phi}^{n-1}$, and we have made use of the following estimate
\begin{align*}
    ( \hat{\Phi}^k ) ^3 & = \big( \Phi_N^k + \dt^2 {\cal P}_N( \Phi^{(2)}_{\dt})^k \big)^3 \nonumber \\
  & = ( \Phi_N^k )^3  + 3 \dt^2 ( \Phi_N^k )^2 {\cal P}_N( \Phi^{(2)}_{\dt} )^k + O (\dt^4 + h^m) \nonumber \\
  & = ( \Phi_N^k )^3  + 3 \dt^2 {\cal P}_N \big(( \Phi_N^k )^2 {\cal P}_N( \Phi^{(2)}_{\dt} )^k\big) + O (\dt^4 + h^m), \quad
  k = n, n-1 .
\end{align*}

Moreover, with an application of Fourier pseudo-spectral approximation in space, we obtain the $O (\dt^3 + h^m)$ truncation error estimate for the constructed solution $\hat{\Phi}$:
\begin{align}
 \frac{ \hat{\Phi}^{n+1} - \hat{\Phi}^n }{\dt}
&=
\Delta_N \Bigl( \frac32 ( \hat{\Phi}^n  )^3 - \frac12 ( \hat{\Phi}^{n-1} )^3
 - \breve{\hat{\Phi}}^{n+1/2}
+ A_0 ( \hat{\Phi}^{n+1} - 2 \hat{\Phi}^n + \hat{\Phi}^{n-1} )   \nonumber
\\
  & \quad
  + A_1 \dt ( \hat{\Phi}^{n+1} - \hat{\Phi}^n )
 + \varepsilon^2 {\cal L}_N \Big( \frac34 \hat{\Phi}^{n+1}
 + \frac14 \hat{\Phi}^{n+1} \Big)  \Bigr)
   + \tau_2^{n+1}    \label{consistency-3-2}
\end{align}
with $\| \tau_2^{n+1} \|_{-1,N} \le C (\dt^3 + h^m) $.
   	
Again, the purpose of the higher-order expansion \eqref{consistency-1} is to obtain an $\ell^{\infty}$ bound of  the error function, as well as its temporal derivative, via its $\ell^2$ norm in higher-order accuracy by utilizing an inverse inequality in the spatial discretization. The details will be demonstrated in the later sections. Under the linear refinement constraint $C_1 h \le \dt \le C_2 h$, a careful analysis reveals that
\begin{equation*}
\|  \hat{\Phi} - \Phi_N \|_\infty \le C ( \dt^2 + h^m )  ,
\end{equation*}
because of the Fourier projection estimate \eqref{projection-est-0} and the fact that $\| \Phi^{(2)}_{\dt} \|_\infty \le C$.
Then, if $\dt$ and $h$ are sufficiently small, in particular,
$C_1h\le\dt\le\min\{(4C(1+C_1^{-m}))^{-1},1\}$,
the following bounds are valid:
\begin{align}
	  &
\|  \hat{\Phi} - \Phi_N \|_\infty \le C ( \dt^2 + h^m )  \le \frac14 ,  \, \, \,  \mbox{so that} \, \,
\|  \hat{\Phi} \|_\infty \le \| \Phi_N \|_\infty + \| \hat{\Phi} - \Phi_N \|_\infty
\le C^* + \frac14 < M_0,  \label{consistency-4-2}
\\
  	  &
\Big\|  \frac{\hat{\Phi}^k - \hat{\Phi}^{k-1} }{\dt} - \frac{ \Phi_N^k - \Phi_N^{k-1} }{\dt} \Big\|_\infty
\le \frac{2C ( \dt^2 + h^m )}{\dt}  \le \frac12 ,  \, \, \,  \mbox{so that} \, \,
\Big\| \frac{\hat{\Phi}^k - \hat{\Phi}^{k-1} }{\dt} \Big\|_\infty \le C^* + \frac12  < M_0 .
  \label{consistency-4-3}
\end{align}

\subsection{A rough error estimate}

Instead of a direct comparison between the numerical solution and the Fourier projection $\Phi_N$ of the exact solution, we analyze the error between the numerical solution and the constructed solution to obtain a higher-order convergence in the $\ell^2$ norm. The following error function is introduced:
	\begin{equation*}
\hat{e}^k := \hat{\Phi}^k - \phi^k .
	\label{error function-2}
	\end{equation*}
Subtracting~\eqref{scheme-2nd-1} from~\eqref{consistency-3-2} gives
	\begin{align}
 \frac{ \hat{e}^{n+1} - \hat{e}^n }{\dt}  &=
\Delta_N \Bigl(  \frac32 ( ( \hat{\Phi}^n  )^3 - (\phi^n)^3 )
   - \frac12 ( ( \hat{\Phi}^{n-1}  )^3 - (\phi^{n-1})^3 )
  - \breve{\hat{e}}^{n+1/2}
+ A_0 ( \hat{e}^{n+1} - 2 \hat{e}^n + \hat{e}^{n-1} )  \nonumber
\\
  & \quad  + A_1 \dt (\hat{e}^{n+1} - \hat{e}^n)
 + \varepsilon^2 {\cal L}_N \Big( \frac34 \hat{e}^{n+1}  + \frac14 \hat{e}^{n-1} \Big)  \Bigr) 
   + \tau_2^{n+1}
   \label{convergence-1}
	\end{align}
with $\breve{\hat{e}}^{n+1/2} := \frac32 \hat{e}^n - \frac12 \hat{e}^{n-1}$. To carry out the nonlinear error estimate, we have to make an $\ell^2$ assumption for the numerical error function at the previous time steps $t_n$, $t_{n-1}$:
\begin{equation}
  \| \hat{e}^k \|_2 \le \dt^{\frac{5}{2}} + h^{m-\frac12}  , \quad  k= n , n-1 . 
   \label{a priori-1}
\end{equation}
Since $C_1 h \le \dt \le C_2 h$, an application of the inverse inequality reveals that
\begin{equation}
  \| \hat{e}^k \|_\infty \le \frac{ \Cinv \| \hat{e}^k \|_2 }{h}
  \le \Cinv'( \dt^{\frac{3}{2}} + h^{m-\frac32} ) , \quad k= n , n-1 ,  \label{a priori-2}
\end{equation}
where $\Cinv$ is the constant in the inverse inequality and $\Cinv'$ depends on $\Cinv$ and $C_2$. Therefore, if $\dt$ and $h$ are sufficiently small, in particular, $C_1h\le\dt\le\min\{(4\Cinv'(1+C_1^{3/2-m}))^{-2},M_0^{-1}\}$,
the $\ell^\infty$ bounds for the numerical solutions at $t_n$ and $t_{n-1}$, as well as their discrete temporal derivatives, become available (for $k=n, n-1$):
\begin{align}
  \| \phi^k \|_\infty &= \| \hat{\Phi}^k - \hat{e}^k \|_\infty
  \le \| \hat{\Phi}^k \|_\infty + \| \hat{e}^k \|_\infty
  \le C^* + \frac14 + \frac14 < M_0 ,   \label{a priori-3-1}
\\
  \Big\| \frac{ \phi^k - \phi^{k-1} }{\dt} \Big\|_\infty   &\le
  \Big\|  \frac{\hat{\Phi}^k - \hat{\Phi}^{k-1} }{\dt}  \Big\|_\infty
  + \Big\|  \frac{ \hat{e}^k - \hat{e}^{k-1} }{\dt}  \Big\|_\infty  \nonumber
\\
  &  \le C^* + \frac12 +\frac{2\Cinv'( \dt^{\frac{3}{2}} + h^{m-\frac32} )}{\dt}
\le  C^* + \frac12 + \frac12 = M_0 ,   \label{a priori-3-2}
\\
  \| \breve{\phi}^{k+1/2} \|_\infty &= \| \phi^k \|_\infty + \frac12 \| \phi^k - \phi^{k-1} \|_\infty
  \le C^* + \frac12 + \frac12 M_0 \dt \le C^* + 1 = M_0 ,  \label{a priori-3-3}
\end{align}
in which the estimates \eqref{consistency-4-2} and \eqref{consistency-4-3}  for $\| \hat{\Phi}^k \|_\infty$ have been recalled. Also, a careful estimate
\[
\dt^{\frac{3}{2}} + h^{m-\frac32}
\le \dt^{\frac{3}{2}} + C_1^{\frac32-m} \dt^{m-\frac32}
= (1 + C_1^{\frac32-m} \dt^{m-3}) \cdot \dt^{\frac12} \cdot \dt
\le \frac{1}{4} \dt
\]
is taken in the derivation of \eqref{a priori-3-2}, where the condition $m\ge3$ is used.
The \emph{a priori} assumption \eqref{a priori-1} will be recovered in the convergence estimate presented later.

Since $\overline{\hat{e}^k}=0$ for any $k \ge 0$, $(-\Delta_N)^{-1} \hat{e}^k$ is well-defined. Taking a discrete inner product with \eqref{convergence-1} by $(-\Delta_N)^{-1} ( \hat{e}^{n+1} - \hat{e}^n)$ leads to
\begin{align}
	&
  \frac{1}{\dt} \| \hat{e}^{n+1} - \hat{e}^n  \|_{-1,N}^2
   + A_0 \langle \hat{e}^{n+1} - 2 \hat{e}^n + \hat{e}^{n-1} ,
 \hat{e}^{n+1} - \hat{e}^n \rangle  + A_1 \dt \| \hat{e}^{n+1} - \hat{e}^n \|_2^2 \nonumber
\\
  & \qquad =
    - \Big\langle  \frac32 ( ( \hat{\Phi}^n  )^3 - (\phi^n)^3 )
   - \frac12 ( ( \hat{\Phi}^{n-1}  )^3 - (\phi^{n-1})^3 )   ,
     \hat{e}^{n+1} - \hat{e}^n \Big\rangle
  +  \langle \breve{\hat{e}}^{n+1/2} , \hat{e}^{n+1} - \hat{e}^n \rangle  \nonumber
\\
  & \qquad\quad
   -  \varepsilon^2 \Big\langle {\cal L}_N \Big( \frac34 \hat{e}^{n+1} + \frac14 \hat{e}^{n-1} \Big) ,
     \hat{e}^{n+1} - \hat{e}^n \Big\rangle
  +  \langle (-\Delta_N)^{-1} ( \hat{e}^{n+1} - \hat{e}^n ) ,
   \tau_2^{n+1} \rangle .
  \label{convergence-2}
\end{align}
For the artificial regularization term on the left-hand side, the following identity is obvious:
\begin{equation}
   \langle \hat{e}^{n+1} - 2 \hat{e}^n + \hat{e}^{n-1} ,
 \hat{e}^{n+1} - \hat{e}^n \rangle
 =  \frac12 ( \| \hat{e}^{n+1} - \hat{e}^n \|_2^2 - \| \hat{e}^n - \hat{e}^{n-1} \|_2^2
  + \| \hat{e}^{n+1} - 2 \hat{e}^n + \hat{e}^{n-1} \|_2^2 ) .   \label{convergence-3}
\end{equation}
The right-hand side term associated with the truncation error can be bounded by
\begin{align*}
    \langle (-\Delta_N)^{-1} (\hat{e}^{n+1} - \hat{e}^n) , \tau_2^{n+1} \rangle
  &\le \| \hat{e}^{n+1} - \hat{e}^n \|_{-1,N} \cdot \| \tau_2^{n+1} \|_{-1,N}  \nonumber
\\
  &\le
  \frac{1}{4 \dt} \| \hat{e}^{n+1} - \hat{e}^n \|_{-1,N}^2
  + \dt \| \tau_2^{n+1} \|_{-1,N}^2  .  
\end{align*}
For the second linear term on the right-hand side, a direct calculation gives
\begin{align}
   \langle \breve{\hat{e}}^{n+1/2} , \hat{e}^{n+1} - \hat{e}^n \rangle
  &= \frac12 \langle \hat{e}^{n+1} + \hat{e}^n  , \hat{e}^{n+1} - \hat{e}^n \rangle
    - \frac12 \langle \hat{e}^{n+1} - 2 \hat{e}^n + \hat{e}^{n-1} ,
    \hat{e}^{n+1} - \hat{e}^n \rangle    \nonumber
\\
  &=
   \frac12 ( \| \hat{e}^{n+1} \|_2^2 - \| \hat{e}^n \|_2^2  )
      - \frac14 ( \| \hat{e}^{n+1} - \hat{e}^n \|_2^2 - \| \hat{e}^n - \hat{e}^{n-1} \|_2^2
  + \| \hat{e}^{n+1} - 2 \hat{e}^n + \hat{e}^{n-1} \|_2^2 )  ,
  \label{convergence-5}
\end{align}
in which the first step comes from the fact that $\breve{\hat{e}}^{n+1/2} = \frac12 (\hat{e}^{n+1} + \hat{e}^n ) - \frac12 ( \hat{e}^{n+1} - 2 \hat{e}^n + \hat{e}^{n-1})$. The nonlocal linear term on the right-hand side can be rewritten as
\begin{align}
  & \quad~
  - \varepsilon^2 \Big\langle {\cal L}_N \Big( \frac34 \hat{e}^{n+1} + \frac14 \hat{e}^{n-1} \Big) ,
  \hat{e}^{n+1} - \hat{e}^n \Big\rangle  \nonumber
\\
  &= - \varepsilon^2 \Big\langle (J * 1) \Big( \frac34 \hat{e}^{n+1} + \frac14 \hat{e}^{n-1} \Big)
  - J * \Big( \frac34 \hat{e}^{n+1} + \frac14 \hat{e}^{n+1} \Big)  ,
   \hat{e}^{n+1} - \hat{e}^n \Big\rangle   \nonumber
\\
  &=
    - \varepsilon^2 (J * 1) \Big\langle  \frac34 \hat{e}^{n+1} + \frac14 \hat{e}^{n-1} ,
   \hat{e}^{n+1} - \hat{e}^n \Big\rangle
   + \varepsilon^2 \Big\langle J * \Big( \frac34 \hat{e}^{n+1} + \frac14 \hat{e}^{n-1} \Big) ,
    \hat{e}^{n+1} - \hat{e}^n \Big\rangle  .  \label{convergence-6-1}
\end{align}
For the first term appearing in the expansion \eqref{convergence-6-1}, the following identity is available:
\begin{align*}
   \Big\langle  \frac34 \hat{e}^{n+1} + \frac14 \hat{e}^{n-1}  ,
   \hat{e}^{n+1} - \hat{e}^n \Big\rangle
   & = \frac12 ( \| \hat{e}^{n+1} \|_2^2 - \| \hat{e}^n \|_2^2  )
   + \frac18 ( \| \hat{e}^{n+1} - \hat{e}^n \|_2^2
   - \| \hat{e}^n - \hat{e}^{n-1} \|_2^2 ) \nonumber
\\
  & \quad
  + \frac18 \| \hat{e}^{n+1} - 2 \hat{e}^n + \hat{e}^{n-1} \|_2^2   .
\end{align*}
Meanwhile, for the term $\varepsilon^2 \langle J * ( \frac34 \hat{e}^{n+1} + \frac14 \hat{e}^{n-1} ) ,  \hat{e}^{n+1} - \hat{e}^n \rangle$, we apply~\eqref{lem 1:0} in Lemma~\ref{lem:1} and obtain
\begin{align}
  & \quad~
   \varepsilon^2 \Big\langle J * \Big( \frac34 \hat{e}^{n+1} + \frac14 \hat{e}^{n-1} \Big)  ,
   \hat{e}^{n+1} - \hat{e}^n \Big\rangle  \nonumber
\\
   &= - \varepsilon^2 \Big\langle J * \Big( \frac34 \hat{e}^{n+1} + \frac14 \hat{e}^{n-1} \Big)   ,
   \Delta_N ( (-\Delta_N)^{-1} ( \hat{e}^{n+1} - \hat{e}^n ) ) \Big\rangle  \nonumber
\\
  &\le
    C_3 \dt \Big\| \frac34 \hat{e}^{n+1} + \frac14 \hat{e}^{n-1} \Big\|^2_2 + \frac{1}{4 \dt}
   \| \nabla_N (-\Delta_N)^{-1} ( \hat{e}^{n+1} - \hat{e}^n ) \|^2_2  \nonumber
\\
  &\le
   \frac{C_3}{8} \dt ( 9 \| \hat{e}^{n+1} \|^2_2 +  \| \hat{e}^{n-1} \|_2^2 )
   + \frac{1}{4 \dt} \| \hat{e}^{n+1} - \hat{e}^n \|_{-1,N}^2   ,
   \label{convergence-6-3}
\end{align}
with $C_3$ depending only on $J$ and $\varepsilon$. Subsequently, a combination of~\eqref{convergence-6-1}--\eqref{convergence-6-3} yields
\begin{align}
  & \quad~
  - \varepsilon^2 \Big\langle {\cal L}_N \Big( \frac34 \hat{e}^{n+1} + \frac14 \hat{e}^{n-1} \Big) ,
  \hat{e}^{n+1} - \hat{e}^n \Big\rangle  \nonumber
\\
  &\le -\frac{\varepsilon^2}{2} (J * 1) ( \| \hat{e}^{n+1} \|_2^2 - \| \hat{e}^n \|_2^2  )
   - \frac{\varepsilon^2}{8} (J * 1) ( \| \hat{e}^{n+1} - \hat{e}^n \|_2^2
   - \| \hat{e}^n - \hat{e}^{n-1} \|_2^2 )  \nonumber
\\
  &\quad
  - \frac{\varepsilon^2}{8} (J * 1)
   \| \hat{e}^{n+1} - 2 \hat{e}^n + \hat{e}^{n-1} \|_2^2  )
    + \frac{C_3}{8} \dt ( 9 \| \hat{e}^{n+1} \|^2_2 +  \| \hat{e}^{n-1} \|_2^2 )
    + \frac{1}{4 \dt}  \| \hat{e}^{n+1} - \hat{e}^n \|_{-1,N}^2  .  \label{convergence-6-4}
\end{align}
For the nonlinear inner product on the right-hand side of \eqref{convergence-2},
we begin with the following nonlinear expansion:
\begin{equation*}
 ( \hat{\Phi}^k  )^3 - ( \phi^k  )^3
 = (  ( \hat{\Phi}^k )^2
    + \hat{\Phi}^k \phi^k  + ( \phi^k  )^2 ) \hat{e}^k , \quad
    k= n , n-1 . 
\end{equation*}
Denote ${\cal C}^k := ( \hat{\Phi}^k )^2 + \hat{\Phi}^k \phi^k  + ( \phi^k  )^2$. The consistency estimate \eqref{consistency-4-2} and \emph{a priori} estimate \eqref{a priori-3-1} indicate
\begin{equation}
  \|  {\cal C}^k  \|_\infty  \le 3  M_0^2 ,  \quad k=n, n-1 .
      \label{convergence-7-3}
\end{equation}
Then we arrive at
\begin{equation*}
 \| ( \hat{\Phi}^k  )^3 - ( \phi^k  )^3  \|_2
 \le \|  {\cal C}^k  \|_\infty \cdot \|  \hat{e}^k  \|_2
 \le 3 M_0^2 \|  \hat{e}^k \|_2  ,  \quad  k=n, n-1 .   \label{convergence-7-4}
\end{equation*}
As a consequence, the following rough estimate can be derived:
\begin{align}
  & \quad~
   - \Big\langle \frac32 ( ( \hat{\Phi}^n  )^3 - (\phi^n)^3 )
   - \frac12 ( ( \hat{\Phi}^{n-1}  )^3 - (\phi^{n-1})^3 ) ,
     \hat{e}^{n+1} - \hat{e}^n \Big\rangle  \nonumber
\\
  &\le
    \Big(  \frac32 \| ( \hat{\Phi}^n  )^3 - (\phi^n)^3  \|_2
   + \frac12 \| ( \hat{\Phi}^{n-1}  )^3 - (\phi^{n-1})^3 \|_2  \Big)
     \cdot   \| \hat{e}^{n+1} - \hat{e}^n \|_2  \nonumber
\\
  &\le
        3 M_0^2  \Big( \frac32 \| \hat{e}^n \|_2 + \frac12  \| \hat{e}^{n-1} \|_2\Big)
     \cdot   \| \hat{e}^{n+1} - \hat{e}^n \|_2    \nonumber
\\
  &\le
      \frac92 M_0^4 ( 9 \|  \hat{e}^n \|_2^2 + \| \hat{e}^{n-1} \|_2^2 )
     + \frac{1}{4}    \| \hat{e}^{n+1} - \hat{e}^n \|_2^2 .
  \label{convergence-7-5}
\end{align}

Therefore, a substitution of \eqref{convergence-3}--\eqref{convergence-5}, \eqref{convergence-6-4} and \eqref{convergence-7-5} into \eqref{convergence-2} leads to
\begin{align*}
	& \quad~
     \frac{1}{2 \dt} \| \hat{e}^{n+1} - \hat{e}^n \|_{-1,N}^2
        + \Big(\frac{A_0}{2}+\frac{1}{4}+\frac{\varepsilon^2}{8}(J*1)\Big) ( \| \hat{e}^{n+1} - \hat{e}^n \|_2^2
        - \| \hat{e}^n - \hat{e}^{n-1} \|_2^2 + \| \hat{e}^{n+1} - 2 \hat{e}^n + \hat{e}^{n-1} \|_2^2)
      \nonumber
\\
  &\quad~
   + A_1 \dt  \| \hat{e}^{n+1} - \hat{e}^n \|_2^2
   + \frac12 ( \varepsilon^2 (J * 1) - 1 ) ( \| \hat{e}^{n+1} \|_2^2 - \| \hat{e}^n \|_2^2  )
   - \frac{1}{4}    \| \hat{e}^{n+1} - \hat{e}^n \|_2\nonumber
\\
  &\le
    \frac{C_3}{8} \dt ( 9 \| \hat{e}^{n+1} \|^2_2 +  \| \hat{e}^{n-1} \|_2^2 )
    + \frac92 M_0^4 ( 9 \|  \hat{e}^n \|_2^2 + \| \hat{e}^{n-1} \|_2^2 )
    + \dt \| \tau_2^{n+1} \|_{-1,N}^2 .
\end{align*}
Making use of the condition (d) for the kernel, we see that
\begin{align}
    \frac{\gamma_0}{2} \| \hat{e}^{n+1} \|_2^2
  &\le
     \Big(\frac{A_0}{2}+\frac{1}{4}+\frac{\varepsilon^2}{8}(J*1)\Big) \| \hat{e}^n - \hat{e}^{n-1} \|_2^2
    + \frac{\gamma_0}{2} \| \hat{e}^n \|_2^2
    + \frac{C_3}{8} \dt ( 9 \| \hat{e}^{n+1} \|^2_2 +  \| \hat{e}^{n-1} \|_2^2 )    \nonumber
\\
  & \quad
    + \frac92 M_0^4 ( 9 \|  \hat{e}^n \|_2^2 + \| \hat{e}^{n-1} \|_2^2 )
    + \dt \| \tau_2^{n+1} \|_{-1,N}^2 .
  \label{convergence-rough-2}
\end{align}
Meanwhile, with the application of the \emph{a priori} error estimate \eqref{a priori-1}, we get
	\begin{eqnarray}
    \frac{\gamma_0}{4} \| \hat{e}^{n+1} \|_2^2
   \le C_4 ( \dt^5 + h^{2 m - 1} ) ,
  \label{convergence-rough-3}
	\end{eqnarray}
provided that $\dt\le\min\{2\gamma_0(9C_3)^{-1},1\}$ and $h\le1$. Then, an application of 2-D inverse inequality gives
	\begin{equation*}
  \| \hat{e}^{n+1} \|_\infty  \le \frac{\Cinv \| \hat{e}^{n+1} \|_2 }{h}
   \le \hat {C}_1 ( \dt^{\frac{3}{2}} + h^{m - \frac32} ),
\quad \mbox{with } \hat{C}_1 := \Cinv(C_2+1)\sqrt{\frac{4C_4}{\gamma_0}},
	\label{convergence-rough-4}
	\end{equation*}
under the linear refinement constraint $C_1 h \le \dt \le C_2 h$. Consequently,
if $\dt$ and $h$ are sufficiently small, in particular,
$C_1h\le\dt\le\min\{(4\hat{C}_1(1+C_1^{3/2-m}))^{-2},1\}$,
the following \emph{a priori} bounds are valid:
\begin{align}
   \| \phi^{n+1} \|_\infty &\le
    \| \hat{\Phi}^{n+1} \|_\infty + \| \hat{e}^{n+1} \|_\infty
  \le C^* + \frac14 + \frac14 < M_0  ,   \label{a priori-4-1}
\\
   \Big\| \frac{ \phi^{n+1} - \phi^n }{\dt} \Big\|_\infty   &\le
  \Big\|  \frac{\hat{\Phi}^{n+1} - \hat{\Phi}^n }{\dt}  \Big\|_\infty
  + \Big\|  \frac{ \hat{e}^{n+1} - \hat{e}^n }{\dt}  \Big\|_\infty
  \le C^* + \frac14 + \frac12  < M_0  .   \label{a priori-4}
\end{align}
In fact, these bounds will play a crucial role in the refined error estimate.

\subsection{A refined error estimate}

It is observed that the error estimate \eqref{convergence-7-5} is too rough; as a result, an inductive argument could not be applied to inequality \eqref{convergence-rough-3}. In this subsection, we perform a more refined error estimate for the nonlinear term, under the \emph{a priori} estimate~\eqref{a priori-4}.

We begin with the following rewritten form of the nonlinear error terms:
\begin{align*}
    \frac32 ( ( \hat{\Phi}^n  )^3 - (\phi^n)^3 )
   - \frac12 (( \hat{\Phi}^{n-1}  )^3 - (\phi^{n-1})^3 )
     & = \frac32 {\cal C}^n \hat{e}^n  -  \frac12  {\cal C}^{n-1} \hat{e}^{n-1}   \nonumber
\\
  &=
  {\cal C}^n
  \Big( \frac32 \hat{e}^n  - \frac12 \hat{e}^{n-1} \Big)
  + \frac12 \hat{e}^{n-1} ( {\cal C}^n - {\cal C}^{n-1} )  \nonumber
\\
  &=
     {\cal C}^n  \breve{\hat{e}}^{n+1/2}
  + \frac12 \hat{e}^{n-1} ( {\cal C}^n - {\cal C}^{n-1} )  .
\end{align*}
And also, similar to \eqref{convergence-5}, the following identity is always valid:
\begin{equation*}
   \breve{\hat{e}}^{n+1/2}  ( \hat{e}^{n+1} - \hat{e}^n )
=
   \frac12 ( ( \hat{e}^{n+1} )^2 - ( \hat{e}^n )^2  )
      - \frac14 ( ( \hat{e}^{n+1} - \hat{e}^n )^2 - ( \hat{e}^n - \hat{e}^{n-1} )^2
  +  ( \hat{e}^{n+1} - 2 \hat{e}^n + \hat{e}^{n-1} )^2 )  .
  \label{convergence-8-1}
\end{equation*}
This in turn leads to the following rewritten form
\begin{align}
    &\quad~
    \Big\langle   \frac32 ( ( \hat{\Phi}^n  )^3 - (\phi^n)^3 )
   - \frac12 ( ( \hat{\Phi}^{n-1}  )^3 - (\phi^{n-1})^3 )  ,
     \hat{e}^{n+1} - \hat{e}^n \Big\rangle   \nonumber
\\
   &=
     \frac12  \langle {\cal C}^n , ( \hat{e}^{n+1} )^2 \rangle
     - \frac12 \langle {\cal C}^n , ( \hat{e}^n )^2 \rangle
    - \frac14  \langle {\cal C}^n , ( \hat{e}^{n+1}  - \hat{e}^n )^2 \rangle
     + \frac14 \langle {\cal C}^n , ( \hat{e}^n - \hat{e}^{n-1} )^2 \rangle     \nonumber
\\
  &\quad
    - \frac14 \langle {\cal C}^n ,
    ( \hat{e}^{n+1}  - 2 \hat{e}^n + \hat{e}^{n+1} )^2 \rangle
  + \frac12 \langle   ( {\cal C}^n - {\cal C}^{n-1} )  \hat{e}^{n-1} ,
     \hat{e}^{n+1} - \hat{e}^n \rangle .
   \label{convergence-8-2}
\end{align}
For the fifth term appearing in the expansion of \eqref{convergence-8-2},
we apply the $\ell^\infty$ bound \eqref{convergence-7-3} and get
\begin{align}
    - \frac14  \langle {\cal C}^n ,
    ( \hat{e}^{n+1}  - 2 \hat{e}^n + \hat{e}^n )^2 \rangle
    &\ge - \frac14 \| {\cal C}^n \|_\infty
     \| \hat{e}^{n+1}  - 2 \hat{e}^n + \hat{e}^{n+1} \|_2^2
\ge
  - \frac{3 M_0^2}{4}  \|  \hat{e}^{n+1}  - 2 \hat{e}^n + \hat{e}^{n-1} \|_2^2  .
   \label{convergence-8-3}
\end{align}
In addition, we have the following $\ell^\infty$ estimate
\begin{align}
  & \quad~
  \| {\cal C}^{k+1} - {\cal C}^k \|_\infty
  =   \|  ( \hat{\Phi}^{k+1} )^2  - ( \hat{\Phi}^k )^2
    +  \hat{\Phi}^{k+1} \phi^{k+1} -  \hat{\Phi}^k \phi^k
    +  ( \phi^{k+1} )^2 - ( \phi^k  )^2 \|_\infty \nonumber
\\
 & \le
  ( \| \hat{\Phi}^{k+1}  \|_\infty + \| \phi^{k+1}  \|_\infty + \| \hat{\Phi}^k  \|_\infty )
   \|  \hat{\Phi}^{k+1} - \hat{\Phi}^k  \|_\infty
   + ( \| \phi^{k+1}  \|_\infty + \| \phi^k  \|_\infty + \| \hat{\Phi}^k  \|_\infty )
   \|  \phi^{k+1} - \phi^k  \|_\infty \nonumber
\\
 & \le
  3 M_0 \cdot M_0 \dt + 3M_0 \cdot M_0 \dt = 6 M_0^2 \dt ,
\label{convergence-8-4}
\end{align}
for $k=n, n-1$, in which the consistency estimates \eqref{consistency-4-2}, \eqref{consistency-4-3}, and the rough bound estimates \eqref{a priori-4-1}, \eqref{a priori-4} have been applied in the second inequality. As a direct consequence, the following lower bound for the last term appearing in~\eqref{convergence-8-2} becomes available:
\begin{align}
   \frac12 \langle   ( {\cal C}^n - {\cal C}^{n-1} )  \hat{e}^{n-1} ,
     \hat{e}^{n+1} - \hat{e}^n \rangle
  &   \ge - \frac12  \|  {\cal C}^n - {\cal C}^{n-1}  \|_\infty  \cdot \|  \hat{e}^{n-1} \|_2
     \cdot \| \hat{e}^{n+1} - \hat{e}^n  \|_2 \nonumber
\\
 & \ge
  - \frac12 \cdot 6 M_0^2 \dt \cdot \|  \hat{e}^{n-1} \|_2
     \cdot \| \hat{e}^{n+1} - \hat{e}^n  \|_2 \nonumber
\\
   &  \ge - \frac32 M_0^2 \dt ( \|  \hat{e}^{n-1} \|_2^2
     + \| \hat{e}^{n+1} - \hat{e}^n  \|_2^2 )  .
   \label{convergence-8-5}
\end{align}
We introduce the quantities
\begin{equation*}
   I_{nl}^k :=  \frac12 \langle {\cal C}^k ,
    (  \hat{e}^k )^2  \rangle  , \quad
  I_{nl,(2)}^k := \frac14 \langle {\cal C}^k ,
  ( \hat{e}^k - \hat{e}^{k-1} )^2 \rangle , \quad k = n, n+1 .  \label{convergence-8-6}
\end{equation*}
It is observed that the first and third terms in~\eqref{convergence-8-2}, are not $I_{nl}^{n+1}$ and $I_{nl, (2)}^{n+1}$, due to the inductive nonlinear coefficient functions. To apply the induction analysis in the later steps, we have to bound their difference. Using the preliminary estimate \eqref{convergence-8-4}, we have
\begin{align}
   \frac12 \langle {\cal C}^n , ( \hat{e}^{n+1} )^2 \rangle  - I_{nl}^{n+1}
  & =  \frac12 \langle {\cal C}^n - {\cal C}^{n+1} , ( \hat{e}^{n+1} )^2 \rangle  \nonumber
\\
  & \ge
  - \frac12 \| {\cal C}^{n+1} - {\cal C}^n \|_\infty \cdot \| \hat{e}^{n+1} \|_2^2
  \ge - 3 M_0^2 \dt \| \hat{e}^{n+1} \|_2^2  ,   \label{convergence-8-7-1}
\\
   - \frac14 \langle {\cal C}^n , ( \hat{e}^{n+1} - \hat{e}^n )^2 \rangle  + I_{nl, (2)}^{n+1}
  & =  \frac14 \langle {\cal C}^{n+1} - {\cal C}^n , ( \hat{e}^{n+1} - \hat{e}^n )^2 \rangle  \nonumber
\\
  & \ge
  - \frac14 \| {\cal C}^{n+1} - {\cal C}^n \|_\infty \cdot \| \hat{e}^{n+1} - \hat{e}^n \|_2^2
 \ge - \frac32 M_0^2 \dt \| \hat{e}^{n+1} - \hat{e}^n \|_2^2  .    \label{convergence-8-7-2}
\end{align}
A combination of~\eqref{convergence-8-2}, \eqref{convergence-8-3}, \eqref{convergence-8-5}, \eqref{convergence-8-7-1} and \eqref{convergence-8-7-2} yields a refined error estimate
\begin{align}
    & \quad
    \Big\langle   \frac32 ( ( \hat{\Phi}^n  )^3 - (\phi^n)^3 )
   - \frac12 ( ( \hat{\Phi}^{n-1}  )^3 - (\phi^{n-1})^3 )  ,
     \hat{e}^{n+1} - \hat{e}^n \Big\rangle    \nonumber
\\
   &\ge
   I_{nl}^{n+1} - I_{nl}^n - ( I_{nl, (2)}^{n+1} - I_{nl, (2)}^n )
   - 3 M_0^2 \dt  \| \hat{e}^{n+1} \|_2^2
   - 3 M_0^2 \dt  \| \hat{e}^{n+1} - \hat{e}^n \|_2^2
   - \frac32 M_0^2 \dt  \| \hat{e}^{n-1} \|_2^2  \nonumber
\\
  & \quad
  - \frac{3 M_0^2}{4}  \|  \hat{e}^{n+1}  - 2 \hat{e}^n + \hat{e}^{n-1} \|_2^2  .
   \label{convergence-8-13}
\end{align}

As a result, a substitution of \eqref{convergence-3}--\eqref{convergence-5}, \eqref{convergence-6-4} and \eqref{convergence-8-13} into \eqref{convergence-2} results in
\begin{align*}
	& \quad~
     \frac{1}{2 \dt} \| \hat{e}^{n+1} - \hat{e}^n \|_{-1,N}^2
      + \Big(\frac{A_0}{2}+\frac{1}{4}+\frac{\varepsilon^2}{8}(J*1)\Big) ( \| \hat{e}^{n+1} - \hat{e}^n \|_2^2 - \| \hat{e}^n - \hat{e}^{n-1} \|_2^2 )
      \nonumber
\\
  &\quad~
      + A_1 \dt \| \hat{e}^{n+1} - \hat{e}^n \|_2^2
      + \frac12 ( \varepsilon^2 (J * 1) - 1 ) ( \| \hat{e}^{n+1} \|_2^2 - \| \hat{e}^n \|_2^2  )
      \nonumber
\\
  &\quad~
   + \Big(\frac{A_0}{2}+\frac{1}{4}+\frac{\varepsilon^2}{8}(J*1)\Big)
   \| \hat{e}^{n+1} - 2 \hat{e}^n + \hat{e}^{n-1} \|_2^2
   + I_{nl}^{n+1}  - I_{nl}^n  - ( I_{nl, (2)}^{n+1} - I_{nl, (2)}^n )  \nonumber
\\
  &\le
     \frac{3 M_0^2}{4} \|  \hat{e}^{n+1}  - 2 \hat{e}^n + \hat{e}^{n-1} \|_2^2
     + 3 M_0^2 \dt  \| \hat{e}^{n+1} - \hat{e}^n \|_2^2
      \nonumber
\\
  &\quad~
     + \Big( \frac34 C_3 + 3 M_0^2 \Big) \dt \| \hat{e}^{n+1} \|^2_2
     + \Big( \frac14 C_3 + \frac32 M_0^2 \Big) \dt \| \hat{e}^{n-1} \|^2_2
    + \dt \| \tau_2^{n+1} \|_{-1,N}^2 .
\end{align*}
Using the condition (d) for the kernel and the condition \eqref{condition-A-1} for the parameter $A_0$ (which indicates that $\frac{A_0}{2}+\frac{1}{4} \ge \frac{3 M_0^2}{4}$), we get
\begin{align*}
	& \quad~
       \frac{\gamma_0}{2} ( \| \hat{e}^{n+1} \|_2^2 - \| \hat{e}^n \|_2^2 ) 	
      + \Big(\frac{A_0}{2}+\frac{1}{4}+\frac{\varepsilon^2}{8}(J*1)\Big) ( \| \hat{e}^{n+1} - \hat{e}^n \|_2^2
      - \| \hat{e}^n - \hat{e}^{n-1} \|_2^2 ) \nonumber
\\
& \quad~
      + I_{nl}^{n+1}  - I_{nl}^n  - ( I_{nl, (2)}^{n+1} - I_{nl, (2)}^n )
      \nonumber
\\
 & \le
     3 M_0^2 \dt  \| \hat{e}^{n+1} - \hat{e}^n \|_2^2
     + \Big( \frac34 C_3 + 3 M_0^2 \Big) \dt \| \hat{e}^{n+1} \|^2_2
     + \Big( \frac14 C_3 + \frac32 M_0^2 \Big) \dt \| \hat{e}^{n-1} \|^2_2
    + \dt \| \tau_2^{n+1} \|_{-1,N}^2 .
\end{align*}
The following quantity is introduced to facilitate the later analysis:
\begin{equation*}
     F^{n+1} : =  \frac{\gamma_0}{2} \| \hat{e}^{n+1} \|_2^2
     + I_{nl}^{n+1}
      + \Big(\frac{A_0}{2}+\frac{1}{4}+\frac{\varepsilon^2}{8}(J*1)\Big) \| \hat{e}^{n+1} - \hat{e}^n \|_2^2
      - I_{nl, (2)}^{n+1} .
  \label{convergence-9-3}
\end{equation*}
In fact, for the last term, we have the following estimate:
\begin{align*}
   I_{nl, (2)}^{n+1} & =  \frac14 \langle {\cal C}^{n+1} ,
  ( \hat{e}^{n+1} - \hat{e}^n )^2 \rangle
  \le \frac14 \| {\cal C}^{n+1} \|_\infty  \|  \hat{e}^{n+1} - \hat{e}^n \|_2^2
\le
     \frac{3 M_0^2 }{4}   \|  \hat{e}^{n+1} - \hat{e}^n \|_2^2
     \le \frac{A_0}{2}  \|  \hat{e}^{n+1} - \hat{e}^n \|_2^2 ,
\end{align*}
in which the $\ell^\infty$ bound for ${\cal C}^{n+1}$ can be obtained in a similar way as in~\eqref{convergence-7-3}. The condition \eqref{condition-A-1} for $A_0$ has been applied as well. This in turn implies that
\begin{equation*}
     F^{n+1} \ge  \frac{\gamma_0}{2} \| \hat{e}^{n+1} \|_2^2
     + I_{nl}^{n+1}
      + \Big(\frac{1}{4}+\frac{\varepsilon^2}{8}(J*1)\Big) \| \hat{e}^{n+1} - \hat{e}^n \|_2^2 \ge 0.
  \label{convergence-9-3-3}
\end{equation*}
As a consequence, the following estimate can be derived:
\begin{equation*}
    F^{n+1} - F^n  \le C_5 \dt F^{n+1}
    + \dt \| \tau_2^{n+1} \|_{-1,N}^2 ,  \quad
    C_5 = \max \{ 2 ( C_3 + 9 M_0^2 ) \gamma_0^{-1}  , 12 M_0^2 \} .
     \label{convergence-9-4}
\end{equation*}
Subsequently, if $\dt\le(2C_5)^{-1}$, an application of the discrete Gronwall's inequality gives the desired convergence estimate:
\begin{equation}
    F^{n+1} \le \hat{C}_2 ( \dt^6 + h^{2m} )  ,
   \label{convergence-9-5}
\end{equation}
due to the fact $\| \tau_2^k \|_{-1,N} \le C (\dt^3 + h^m)$ for $k \le n+1$. In particular, the following bound is observed:
\begin{equation}
    \| \hat{e}^{n+1}  \|_2
     \le \sqrt{2\hat{C}_2\gamma_0^{-1}} ( \dt^3 + h^m )  \le \dt^{\frac{5}{2}} + h^{m-\frac12} 
   \label{convergence-9-6}
\end{equation}
for $\dt\le(2\hat{C}_2)^{-1}\gamma_0$ and $h\le(2\hat{C}_2)^{-1}\gamma_0$,
so that the \emph{a priori} assumption \eqref{a priori-1} has been recovered at time instant $t_{n+1}$. Therefore, the analysis can be carried out in the induction style. This completes the error estimate for $\hat{e}$, the numerical error between the numerical solution $\phi$ and the constructed approximation solution $\hat{\Phi}$.

Of couse, the error estimate \eqref{convergence-0} becomes a direct consequence of the following identity
\begin{equation*}
  e^k = \hat{e}^k - \dt^2  \mathcal{P}_N \Phi^{(2)}_{\dt} , \quad
  \mbox{(by the construction~\eqref{consistency-1})} ,
\end{equation*}
combined with the fact that $\| ( \Phi^{(2)}_{\dt} )^k \|_2 \le C$ for any $k \ge 0$. This completes the proof of Theorem~\ref{thm:convergence}.

\begin{remark}
Since the inverse inequality used in \eqref{a priori-2} depends on the number of dimension,
we briefly illustrate the necessary modifications of the above derivation 
if one considers the three-dimensional case.

Instead of \eqref{a priori-1},
the $\ell^2$ assumption for the induction would be
\begin{equation}
  \| \hat{e}^k \|_2 \le \dt^{\frac{11}{4}} + h^{m-\frac14}  , \quad  k= n , n-1 .
   \label{a priori-1-3d}
\end{equation}
Then, under the requirement $C_1 h \le \dt \le C_2 h$,
an application of the 3-D inverse inequality gives
\begin{equation*}
  \| \hat{e}^k \|_\infty \le \frac{ \Cinv \| \hat{e}^k \|_2 }{h^{\frac{3}{2}}}
  \le C ( \dt^{\frac{5}{4}} + h^{m-\frac74} ) , \quad k= n , n-1 .  \label{a priori-2-3d}
\end{equation*}
The $\ell^\infty$ bounds of $\phi^{k}$ and $(\phi^{k}-\phi^{k-1})/\dt$ with $k=n,n-1$ can be similarly obtained as \eqref{a priori-3-1}--\eqref{a priori-3-3}.
We need to recover the estimate \eqref{a priori-1-3d} for $k=n+1$.
First, a rough error estimate, independent of the number of dimension, leads to \eqref{convergence-rough-2},
and an application of the estimate \eqref{a priori-1-3d} gives
\begin{equation*}
    \frac{\gamma_0}{4} \| \hat{e}^{n+1} \|_2^2
   \le C_4 ( \dt^{\frac{11}{2}} + h^{2 m - \frac12} ).
  \label{convergence-rough-3-3d}
\end{equation*}
With the linear refinement constraint $C_1 h \le \dt \le C_2 h$, applying the 3-D inverse inequality gives
	\begin{equation*}
  \| \hat{e}^{n+1} \|_\infty  \le \frac{\Cinv \| \hat{e}^{n+1} \|_2 }{h^{\frac32}}
   \le \hat {C}_1 ( \dt^{\frac{5}{4}} + h^{m - \frac74} ),
	\label{convergence-rough-4}
	\end{equation*}
so that the $\ell^\infty$ bounds for $\phi^{n+1}$ and $(\phi^{n+1}-\phi^n)/\dt$ can be derived as \eqref{a priori-4-1} and \eqref{a priori-4}.
Second, a refined error estimate can be performed to obtain \eqref{convergence-9-5},
and the estimate \eqref{convergence-9-6} needs to be replaced by
\begin{equation*}
    \| \hat{e}^{n+1}  \|_2
     \le \sqrt{\frac{2\hat{C}_2}{\gamma_0}} ( \dt^3 + h^m )  \le \dt^{\frac{11}{4}} + h^{m-\frac14} ,
   \label{convergence-9-6-3d}
\end{equation*}
so that the assumption \eqref{a priori-1-3d} is recovered at time instant $t_{n+1}$.
\end{remark}

\section{Energy stability analysis}
\label{sect_stability}

The following energy stability estimate can be established with respect to a modified energy.

\begin{theorem}
\label{thm:energy stab}
Under the assumptions of Theorem~\ref{thm:convergence},
if $A_0$, $A_1$ and $\dt$ satisfy
\begin{equation}
  A_0 \ge \frac32 M_0^2 ,  \quad  A_1 \ge  \frac{19}{4} M_0^2 , \quad
   C_J \varepsilon^2 \dt \le 2 (J * 1)
  \label{condition-A0-A1}
\end{equation}
with $C_J$ depending only on $J$,
we have a modified energy dissipation property for \eqref{scheme-2nd-1} as follows:
\[
\tilde{E}_N (\phi^{n+1}, \phi^n, \phi^{n-1}) \le \tilde{E}_N (\phi^n, \phi^{n-1}, \phi^{n-2}),
\]
where
\begin{align}
   \tilde{E}_N (\phi^{n+1}, \phi^n, \phi^{n-1}) &  := E_N (\phi^{n+1})
  + \frac{27}{8} M_0^2  \dt \| \phi^{n+1} - \phi^n \|_2^2
    + \Big(\frac{A_0}{2}+\frac{1}{4}+\frac{\varepsilon^2}{8}(J*1)\Big)
      \| \phi^{n+1} - \phi^n \|_2^2\nonumber
\\
  & \quad~
    - \frac14 \langle ( \phi^{n+1/2} )^2 + \phi^{n+1/2} \breve{\phi}^{n+1/2}
    + (\breve{\phi}^{n+1/2} )^2,   ( \phi^{n+1} - \phi^n )^2 \rangle.
\label{energy stab-0}
\end{align}
\end{theorem}

\begin{proof}
Taking a discrete inner product with~\eqref{scheme-2nd-1} by $(-\Delta_N)^{-1} (\phi^{n+1} - \phi^n)$ yields
\begin{align}
	& \quad~
  \frac{1}{\dt} \| \phi^{n+1} - \phi^n \|_{-1, N}^2
 + A_0 \langle \phi^{n+1} - 2 \phi^n + \phi^{n-1} ,
 \phi^{n+1} - \phi^n \rangle  + A_1 \dt  \| \phi^{n+1} - \phi^n \|_2^2 \nonumber
\\
  &=
     - \Big\langle \frac32 ( \phi^n )^3 -  \frac12 ( \phi^{n-1} )^3 ,
     \phi^{n+1} - \phi^n \Big\rangle
   + \langle \breve{\phi}^{n+1/2} , \phi^{n+1} - \phi^n \rangle
   - \varepsilon^2 \Big\langle {\cal L}_N \Big( \frac34 \phi^{n+1} + \frac14 \phi^{n-1} \Big) ,
  \phi^{n+1} - \phi^n \Big\rangle .   \label{energy stab-1}
\end{align}
For the artificial regularization term, the following identity is straightforward:
\begin{align}
\langle \phi^{n+1} - 2 \phi^n + \phi^{n-1} , \phi^{n+1} - \phi^n \rangle
=  \frac12 ( \| \phi^{n+1} - \phi^n \|_2^2 - \| \phi^n - \phi^{n-1} \|_2^2
  + \| \phi^{n+1} - 2 \phi^n + \phi^{n-1} \|_2^2 ) .
 \label{energy stab-3}
\end{align}
For the second linear term on the right-hand side, we see that
\begin{align}
   \langle \breve{\phi}^{n+1/2} , \phi^{n+1} - \phi^n \rangle
  &= \frac12 \langle \phi^{n+1} + \phi^n , \phi^{n+1} - \phi^n \rangle
    - \frac12 \langle \phi^{n+1} - 2 \phi^n + \phi^{n-1} ,
    \phi^{n+1} - \phi^n \rangle    \nonumber
\\
  &=
   \frac12 ( \| \phi^{n+1} \|_2^2 - \| \phi^n \|_2^2 )
      - \frac14 ( \| \phi^{n+1} - \phi^n \|_2^2 - \| \phi^n - \phi^{n-1} \|_2^2
  + \| \phi^{n+1} - 2 \phi^n + \phi^{n-1} \|_2^2 )  ,
  \label{energy stab-4}
\end{align}
in which the first step comes from the fact $\breve{\phi}^{n+1/2} = \frac12 ( \phi^{n+1} + \phi^n)  - \frac12 ( \phi^{n+1} - 2 \phi^n + \phi^{n-1})$. For the nonlocal diffusion term on the right-hand side, we rewrite it as
\begin{align}
  & \quad
  - \varepsilon^2 \Big\langle {\cal L}_N \Big( \frac34 \phi^{n+1} + \frac14 \phi^{n-1} \Big) ,
  \phi^{n+1} - \phi^n \Big\rangle  \nonumber
\\
  &= - \varepsilon^2 \Big\langle (J * 1) \Big( \frac34 \phi^{n+1} + \frac14 \phi^{n-1} \Big)
  - J * \Big( \frac34 \phi^{n+1} + \frac14 \phi^{n+1} \Big)  ,
   \phi^{n+1} - \phi^n \Big\rangle   \nonumber
\\
  &=
    - \varepsilon^2 (J * 1) \Big\langle  \frac34 \phi^{n+1} + \frac14 \phi^{n-1} ,
   \phi^{n+1} - \phi^n \Big\rangle
   + \varepsilon^2 \Big\langle J * \Big( \frac34 \phi^{n+1} + \frac14 \phi^{n-1} \Big) ,
    \phi^{n+1} - \phi^n \Big\rangle  .  \label{energy stab-5-1}
\end{align}
For the first term appearing in  \eqref{energy stab-5-1}, we have
\begin{align*}
   \Big\langle  \frac34 \phi^{n+1} + \frac14 \phi^{n-1}  ,
   \phi^{n+1} - \phi^n \Big\rangle
  &  = \frac12 ( \| \phi^{n+1} \|_2^2 - \| \phi^n \|_2^2  )
   + \frac18 ( \| \phi^{n+1} - \phi^n \|_2^2
   - \| \phi^n - \phi^{n-1} \|_2^2 ) \nonumber
\\
  & \quad
  + \frac18 \| \phi^{n+1} - 2 \phi^n + \phi^{n-1} \|_2^2   .
\end{align*}
Meanwhile, for the second term, we apply~\eqref{lem 1:0} in Lemma~\ref{lem:1} and obtain
\begin{align}
  & \quad~
   \varepsilon^2 \Big\langle J * \Big( \frac34 \phi^{n+1} + \frac14 \phi^{n-1} \Big)  ,
   \phi^{n+1} - \phi^n \Big\rangle  \nonumber
\\
   & = \frac{\varepsilon^2}{2} \langle J * ( \phi^{n+1} + \phi^n )   ,
    \phi^{n+1} - \phi^n  \rangle
   - \frac{\varepsilon^2}{4} \langle J * ( \phi^{n+1} - 2 \phi^n + \phi^{n-1} )   ,
   \Delta_N ( (-\Delta_N)^{-1} ( \phi^{n+1} - \phi^n ) ) \rangle  \nonumber
\\
  &\le
    \frac{\varepsilon^2}{2} ( \langle J * \phi^{n+1} , \phi^{n+1} \rangle
    -   \langle J * \phi^n , \phi^n \rangle)
    + \frac{1}{16} C_J  \varepsilon^4 \dt \| \phi^{n+1} - 2 \phi^n + \phi^{n-1} \|^2_2 + \frac{1}{\dt}
   \| \phi^{n+1} - \phi^n  \|_{-1,N}^2 ,    \label{energy stab-5-3}
\end{align}
where $C_J$ depends only on $J$.
Subsequently, a combination of \eqref{energy stab-5-1}--\eqref{energy stab-5-3} yields
\begin{align}
  & \quad~
   \varepsilon^2 \Big\langle {\cal L}_N \Big( \frac34 \phi^{n+1} + \frac14 \phi^{n-1} \Big) ,
  \phi^{n+1} - \phi^n \Big\rangle  \nonumber
\\
  &\ge  \frac{\varepsilon^2}{2} ( \langle {\cal L}_N \phi^{n+1} , \phi^{n+1}  \rangle
    - \langle {\cal L}_N \phi^n , \phi^n   \rangle )
     + \frac{\varepsilon^2}{8} (J * 1) ( \| \phi^{n+1} - \phi^n \|_2^2
   - \| \phi^n - \phi^{n-1} \|_2^2 )  \nonumber
\\
  & \quad
  + \Big( \frac{\varepsilon^2}{8}  (J * 1)  - \frac{1}{16} C_J \varepsilon^4 \dt \Big)
   \| \phi^{n+1} - 2 \phi^n + \phi^{n-1} \|_2^2
    - \frac{1}{\dt}  \| \phi^{n+1} - \phi^n \|_{-1,N}^2  .  \label{energy stab-5-4}
\end{align}

For the nonlinear inner product,  we begin with the following decomposition:
\begin{align}
  &
  \frac32 ( \phi^n )^3 -  \frac12 ( \phi^{n-1} )^3
  - \frac14 ( ( \phi^{n+1}  )^2 + (\phi^n )^2 ) (\phi^{n+1} + \phi^n ) \nonumber
\\
  = &
  - \frac38 ( 5 \phi^n + \phi^{n-1} ) ( \phi^n - \phi^{n-1} )^2
   - \frac18 ( \phi^{n+1} + \phi^n ) ( \phi^{n+1} - \phi^n )^2 \nonumber
\\
  &
  - \frac12 \Big( ( \phi^{n+1/2} )^2 + \phi^{n+1/2} \breve{\phi}^{n+1/2}
  + (\breve{\phi}^{n+1/2} )^2 \Big) ( \phi^{n+1} - 2 \phi^n + \phi^{n-1} ) , \label{energy stab-6-1}
\end{align}
where $\phi^{n+1/2} = \frac12 (\phi^{n+1} + \phi^n)$ and $\breve{\phi}^{n+1/2} = \frac32 \phi^n - \frac12 \phi^n$. For the first two terms appearing in \eqref{energy stab-6-1}, the following inner product estimates can be derived:
\begin{align}
  &
   - \frac38 \Big\langle ( 5 \phi^n + \phi^{n-1} ) ( \phi^n - \phi^{n-1} )^2  ,
  \phi^{n+1} - \phi^n \Big\rangle  \nonumber
\\
  & \qquad~\ge
  - \Big( \frac{15}{8} \| \phi^n \|_\infty + \frac38 \| \phi^{n-1} \|_\infty \Big)
  \cdot \| \phi^n - \phi^{n-1} \|_\infty \cdot \| \phi^n - \phi^{n-1}  \|_2
  \cdot \| \phi^{n+1} - \phi^n \|_2  \nonumber
\\
  & \qquad~\ge
  - \frac94 M_0 \cdot M_0 \dt  \cdot \| \phi^n - \phi^{n-1}  \|_2
  \cdot \| \phi^{n+1} - \phi^n \|_2 \nonumber
\\
  & \qquad~\ge - \frac98 M_0^2 \dt  ( \| \phi^n - \phi^{n-1}  \|_2^2
  + \| \phi^{n+1} - \phi^n \|_2^2 )  ,   \label{energy stab-6-2}
\\
  &
   - \frac18 \Big\langle ( \phi^{n+1} + \phi^n ) ( \phi^{n+1} - \phi^n )^2  ,
  \phi^{n+1} - \phi^n \Big\rangle   \nonumber
\\
  & \qquad~\ge
  - \frac18 ( \| \phi^{n+1} \|_\infty + \| \phi^n \|_\infty )
  \cdot \| \phi^{n+1} - \phi^n \|_\infty \cdot \| \phi^{n+1} - \phi^n \|_2^2  \nonumber
\\
  & \qquad~\ge
  - \frac14 M_0 \cdot M_0 \dt  \cdot \| \phi^{n+1} - \phi^n \|_2^2
 = - \frac14 M_0^2 \dt  \| \phi^{n+1} - \phi^n \|_2^2  ,  \label{energy stab-6-3}
\end{align}
in which the \emph{a priori} estimates~\eqref{a priori-3-1}, \eqref{a priori-3-2} and the bounds~\eqref{a priori-4-1}, \eqref{a priori-4} have been repeatedly applied. For the last term appearing in~\eqref{energy stab-6-1}, we denote ${\cal C}^{n+1/2} = ( \phi^{n+1/2} )^2 + \phi^{n+1/2} \breve{\phi}^{n+1/2} + (\breve{\phi}^{n+1/2} )^2$. In turn, the following estimates become available:
\begin{align}
  & \quad~
   \| {\cal C}^{n+1/2}  \|_\infty  \le
   \frac32 ( \|\phi^{n+1/2} \|_\infty^2
  + \| \breve{\phi}^{n+1/2} \|_\infty^2 )
  \le \frac32 ( M_0^2 + M_0^2 ) = 3 M_0^2 ,   \label{energy stab-6-4-1}
\\
  & \quad~
  \| {\cal C}^{n+1/2}  - {\cal C}^{n-1/2} \|_\infty    \nonumber
\\
  & \le
   \|  ( \phi^{n+1/2} )^2 - (\phi^{n-1/2} )^2 \|_\infty
  + \| \phi^{n+1/2} \breve{\phi}^{n+1/2} - \phi^{n-1/2} \breve{\phi}^{n-1/2}  \|_\infty
  + \| (\breve{\phi}^{n+1/2} )^2  - (\breve{\phi}^{n-1/2} )^2 \|_\infty  \nonumber
\\
  & \le
  ( \|  \phi^{n+1/2}  \|_\infty + \|\breve{\phi}^{n+1/2}\|_\infty + \| \phi^{n-1/2} \|_\infty )
  \cdot   \|  \phi^{n+1/2}  - \phi^{n-1/2}  \|_\infty    \nonumber
\\
  & \quad
  +  ( \| \breve{\phi}^{n+1/2} \|_\infty + \|\phi^{n-1/2}\|_\infty + \| \breve{\phi}^{n-1/2} \|_\infty )
   \cdot \| \breve{\phi}^{n+1/2}  - \breve{\phi}^{n-1/2} \|_\infty    \nonumber
\\
  & \le
  (M_0 + M_0 + M_0 ) \cdot M_0 \dt
  + ( M_0 + M_0 + M_0 ) \cdot 2 M_0 \dt
  = 9 M_0^2 \dt .  \label{energy stab-6-4-2}
\end{align}
Again, the \emph{a priori} estimates~\eqref{a priori-3-1}--\eqref{a priori-3-3} and the bounds~\eqref{a priori-4-1}, \eqref{a priori-4} have been repeatedly applied. Meanwhile, we introduce $I_{nl, (3)}^{k+1} := \frac14 \langle {\cal C}^{k+1/2} ,   ( \phi^{k+1} - \phi^k )^2 \rangle$. Then we get
\begin{align}
  & \quad~
 - \frac12 \Big\langle \Big( ( \phi^{n+1/2} )^2 + \phi^{n+1/2} \breve{\phi}^{n+1/2}
  + (\breve{\phi}^{n+1/2} )^2 \Big) ( \phi^{n+1} - 2 \phi^n + \phi^{n-1} )  ,
  \phi^{n+1} - \phi^n \Big\rangle  \nonumber 
\\
 & =
  - \frac14 \langle {\cal C}^{n+1/2} ,   ( \phi^{n+1} - \phi^n )^2 - ( \phi^n - \phi^{n-1} )^2
   +  ( \phi^{n+1} - 2 \phi^n + \phi^{n-1} )^2  \rangle  \nonumber
\\
 & =
   - \frac14 \langle {\cal C}^{n+1/2} ,   ( \phi^{n+1} - \phi^n )^2 \rangle
   + \frac14   \langle {\cal C}^{n-1/2} ,  ( \phi^n - \phi^{n-1} )^2  \rangle  \nonumber
\\
  & \quad
   + \frac14   \langle {\cal C}^{n+1/2} - {\cal C}^{n-1/2} ,
    ( \phi^n - \phi^{n-1} )^2  \rangle
  - \frac14  \langle {\cal C}^{n+1/2} ,
  ( \phi^{n+1} - 2 \phi^n + \phi^{n-1} )^2  \rangle .  \label{energy stab-6-4-3}
\end{align}
The last two terms could be bounded as follows:
\begin{align*}
  \frac14   \langle {\cal C}^{n+1/2} - {\cal C}^{n-1/2} ,
    ( \phi^n - \phi^{n-1} )^2  \rangle
   & \ge - \frac14   \| {\cal C}^{n+1/2} - {\cal C}^{n-1/2}  \|_\infty
    \cdot \| \phi^n - \phi^{n-1} \|_2^2     \nonumber
\\
 & \ge
  - \frac14   \cdot 9 M_0^2 \dt  \cdot \| \phi^n - \phi^{n-1} \|_2^2
   = - \frac94 M_0^2 \dt \| \phi^n - \phi^{n-1} \|_2^2  , 
\\
   - \frac14  \langle {\cal C}^{n+1/2} ,
  ( \phi^{n+1} - 2 \phi^n + \phi^{n-1} )^2  \rangle
   & \ge - \frac14   \| {\cal C}^{n+1/2}  \|_\infty
    \cdot \| \phi^{n+1} - 2 \phi^n + \phi^{n-1} \|_2^2     \nonumber
\\
 & \ge
  - \frac34 M_0^2 \| \phi^{n+1} - 2 \phi^n + \phi^{n-1} \|_2^2  , 
\end{align*}
by using the preliminary estimates \eqref{energy stab-6-4-1} and \eqref{energy stab-6-4-2}. Going back~\eqref{energy stab-6-4-3}, we obtain
\begin{align}
  & \quad~
 - \frac12 \Big\langle \Big( ( \phi^{n+1/2} )^2 + \phi^{n+1/2} \breve{\phi}^{n+1/2}
  + (\breve{\phi}^{n+1/2} )^2 \Big) ( \phi^{n+1} - 2 \phi^n + \phi^{n-1} )  ,
  \phi^{n+1} - \phi^n \Big\rangle  \nonumber
\\
  & \ge
    - I_{nl, (3)}^{n+1} + I_{nl, (3)}^n
    - \frac94 M_0^2 \dt \| \phi^n - \phi^{n-1} \|_2^2
    - \frac34 M_0^2 \| \phi^{n+1} - 2 \phi^n + \phi^{n-1} \|_2^2  .  \label{energy stab-6-4-6}
\end{align}
On the other hand, the following estimate is straightforward:
\begin{eqnarray}
 \Big\langle \frac14 ( ( \phi^{n+1}  )^2 + (\phi^n )^2 ) (\phi^{n+1} + \phi^n ) ,
   \phi^{n+1} - \phi^n \Big\rangle
= \frac14 ( \| \phi^{n+1} \|_4^4 - \|  \phi^n \|_4^4 ) .
   \label{energy stab-6-7}
\end{eqnarray}
Therefore, a combination of~\eqref{energy stab-6-2}, \eqref{energy stab-6-3}, \eqref{energy stab-6-4-6}, \eqref{energy stab-6-7} and \eqref{energy stab-6-1} yields
\begin{align}
 \Big\langle \frac32 ( \phi^n  )^3  - \frac12 ( \phi^{n-1} )^3 , \phi^{n+1} - \phi^n \Big\rangle
  & \ge
    \frac14 ( \| \phi^{n+1} \|_4^4 - \|  \phi^n \|_4^4 )
    - I_{nl, (3)}^{n+1} + I_{nl, (3)}^n
    - \frac34 M_0^2 \| \phi^{n+1} - 2 \phi^n + \phi^{n-1} \|_2^2    \nonumber
\\
  & \quad
   -  \frac{11}{8} M_0^2 \dt \| \phi^{n+1} - \phi^n \|_2^2
   -  \frac{27}{8} M_0^2 \dt \| \phi^n - \phi^{n-1} \|_2^2    .  \label{energy stab-6-8}
\end{align}

Finally, a substitution of~\eqref{energy stab-3}, \eqref{energy stab-4}, \eqref{energy stab-5-4} and \eqref{energy stab-6-8} into \eqref{energy stab-1} results in
\begin{align*}
   &
     E_N (\phi^{n+1} )- E_N (\phi^n )
     - I_{nl, (3)}^{n+1} + I_{nl, (3)}^n    \nonumber
\\
   &
      + \Big( \frac{A_0}{2}+\frac{1}{4}+\frac{\varepsilon^2}{8}(J*1) \Big)
      ( \| \phi^{n+1} - \phi^n \|_2^2 - \| \phi^n - \phi^{n-1} \|_2^2 )
      \nonumber
\\
  &
   + \Big( \frac{A_0}{2}+\frac{1}{4}+\frac{\varepsilon^2}{8}(J*1)
     - \frac{1}{16} C_J \varepsilon^4  \dt - \frac34 M_0^2   \Big)
    \| \phi^{n+1} - 2 \phi^n + \phi^{n-1} \|_2^2    \nonumber
\\
  &
   + \Big( A_1 - \frac{11}{8} M_0^2 \Big) \dt \| \phi^{n+1} - \phi^n \|_2^2
   - \frac{27}{8} M_0^2 \dt \| \phi^n - \phi^{n-1} \|_2^2  \le 0 .
\end{align*}
As a result, the constraint \eqref{condition-A0-A1} leads to
\begin{equation*}
  \frac{A_0}{2}+\frac{1}{4}- \frac34 M_0^2 \ge 0 , \quad
  \frac{\varepsilon^2}{8}(J*1)  - \frac{1}{16} C_J \varepsilon^4  \dt \ge 0 , \quad
  A_1 - \frac{11}{8} M_0^2  \ge \frac{27}{8} M_0^2 ,
\end{equation*}
and thus, we obtain a modified energy inequality
\[
   \tilde{E}_N (\phi^{n+1}, \phi^n, \phi^{n-1} )- \tilde{E}_N (\phi^n, \phi^{n-1}, \phi^{n-2} ) \le 0 .
\]
This completes the proof of Theorem \ref{thm:energy stab}.
\end{proof}

\begin{remark}
In the modified energy \eqref{energy stab-0}, we see that although the correction terms include a negative part, $- I_{nl, (3)}^{n+1} = - \frac14 \langle {\cal C}^{n+1/2} ,   ( \phi^{n+1} - \phi^n )^2 \rangle$, the overall correction values are still non-negative. The preliminary estimate \eqref{energy stab-6-4-1} reveals that
\begin{equation}
\label{estimate_I3}
- I_{nl, (3)}^{n+1}
=
    - \frac14 \langle {\cal C}^{n+1/2} ,   ( \phi^{n+1} - \phi^n )^2 \rangle
\ge
    - \frac14 \| {\cal C}^{n+1/2} \|_\infty \| \phi^{n+1} - \phi^n \|_2^2
\ge
    - \frac34 M_0^2  \| \phi^{n+1} - \phi^n \|_2^2  ,
\end{equation}
which in turn gives
\begin{equation*}
    \Big(\frac{A_0}{2}+\frac{1}{4}+\frac{\varepsilon^2}{8}(J*1)\Big)
      \| \phi^{n+1} - \phi^n \|_2^2  \ge   I_{nl, (3)}^{n+1} . 
\end{equation*}
As a consequence, the modified energy dissipation property \eqref{energy stab-0} leads to a uniform-in-time bound for the original energy functional. More precisely, for any $n\ge 2$, we have
\begin{equation}
\label{energybound1}
  E_N (\phi^n) \le  \tilde{E}_N (\phi^n, \phi^{n-1}, \phi^{n-2} ) \le \cdots
  \le  \tilde{E}_N (\phi^2, \phi^1, \phi^0 ),
\end{equation}
where, by \eqref{energy stab-0} and \eqref{estimate_I3},
\begin{align*}
\tilde{E}_N (\phi^2, \phi^1, \phi^0 )
&  = E_N (\phi^2) + \frac{27}{8} M_0^2  \dt \| \phi^2 - \phi^1 \|_2^2
+ \Big(\frac{A_0}{2}+\frac{1}{4}+\frac{\varepsilon^2}{8}(J*1)\Big) \| \phi^2 - \phi^1 \|_2^2
    - \frac14 \langle {\cal C}^{3/2},   ( \phi^2 - \phi^1 )^2 \rangle \nonumber \\
& \le E_N (\phi^2) + \Big(\frac{A_0}{2}+\frac{1}{4}+\frac{\varepsilon^2}{8}(J*1) + \frac34 M_0^2 + \frac{27}{8} M_0^2  \dt\Big) \| \phi^2 - \phi^1 \|_2^2. 
\end{align*}
By conducting similar deductions as done in~\cite{du18}, we can obtain
\[
E_N (\phi^2) \le E_N (\phi^1)
+ \frac{\varepsilon^2}{8} \langle \mathcal{L}_N (\phi^1-\phi^0), \phi^1-\phi^0 \rangle
+ \Big(\frac{A_0}{2}+\frac{1}{4}\Big) \| \phi^1 - \phi^0 \|_2^2 + \frac{4A_0}{3} \| \phi^2 - \phi^1 \|_2^2.
\]
For the nonlocal term, similar to the proof of Lemma \ref{lem:1}~\cite{LiX21a}, we have
\begin{align*}
\frac{\varepsilon^2}{8} \langle \mathcal{L}_N (\phi^1-\phi^0), \phi^1-\phi^0 \rangle
& = \frac{\varepsilon^2}{8}(J*1) \| \phi^1-\phi^0 \|_2^2
- \frac{\varepsilon^2}{8} \langle J * (\phi^1-\phi^0), \phi^1-\phi^0 \rangle \\
& \le \frac{\varepsilon^2}{8}(J*1) \| \phi^1-\phi^0 \|_2^2
+ \frac{\varepsilon^2}{8} \tilde{C}_J |\Omega| \| \phi^1-\phi^0 \|_2^2,
\end{align*}
where $\tilde{C}_J$ depends only on the kernel $J$. Then we obtain
\begin{equation*}
E_N (\phi^2) \le E_N (\phi^1)
+ \Big(\frac{A_0}{2}+\frac{1}{4}+\frac{\varepsilon^2}{8}(J*1)+\frac{\varepsilon^2}{8} \tilde{C}_J |\Omega|\Big)
\| \phi^1 - \phi^0 \|_2^2 + \frac{4A_0}{3} \| \phi^2 - \phi^1 \|_2^2.
\end{equation*}
For the numerical solution $\phi^1$ at time $t=t_1$, by either the discrete gradient scheme or the second-order RK method (discussed at the end of Section \ref{sect_scheme}), the following initial accuracy is available:
\begin{equation}
\label{energybound4}
E_N (\phi^1) \le E_N(\phi^0) + c_1 \dt^2,
\end{equation}
where $c_1$ is independent of $\dt$. Combining \eqref{energybound1}--\eqref{energybound4} and the estimate \eqref{a priori-4}, we arrive at
\[
E_N(\phi^n) \le E_N(\phi^0) + C \dt^2
\]
with $C$ independent of $\dt$. This gives a uniform bound of the original energy functional.
\end{remark}

\begin{remark}
Double stabilization terms, namely $A_0 \Delta_N ( \phi^{n+1} - 2 \phi^n + \phi^{n-1})$ and $A_1 \dt \Delta_N ( \phi^{n+1} - \phi^n )$, have to be included in the modified Crank--Nicolson scheme~\eqref{scheme-2nd-1} to ensure the energy stability estimate, as demonstrated in the proof of Theorem~\ref{thm:energy stab}. Meanwhile, for the modified BDF2 scheme reported in~\cite{LiX21b}, only one stabilization term, $A_0 \Delta_N ( \phi^{n+1} - 2 \phi^n + \phi^{n-1})$, is necessary in the theoretical justification of the energy stability analysis. Such a difference comes from the subtle fact that, the BDF2 temporal discretization brings more numerical diffusion than the Crank--Nicolson approximation, since an inner product with the nonlocal diffusion term by the discrete temporal derivative gives an $O (1)$ coefficient of $\| \phi^{n+1} - \phi^n \|_2^2$ of numerical stabilization in the BDF2 method, while such an inner product yields an almost exact energy identity in the Crank--Nicolson approximation. See the related energy estimates for the BDF2 approaches~\cite{cheng2019a, WangL2018, yan18} and the Crank--Nicolson ones~\cite{cheng16a, diegel17, diegel16, guo16, guo2021}. In particular, for the classic Cahn--Hilliard equation, it turns out that the theoretical estimate has been available for the stabilized BDF2 scheme~\cite{LiD2017, LiD2017b}, while the associated estimate for the Crank--Nicolson one has faced serious difficulties. Also see a related work~\cite{Meng20} for the artificial regularization parameter analysis for the no-slope-selection thin film model.

On the other hand, the modified energy functional for the energy stability estimate reported for the BDF2 scheme~\cite{LiX21b} takes a form of
\begin{equation}
\tilde{E}_N^* (\phi^{n+1}, \phi^n) = E_N (\phi^{n+1})
 + \frac{A_0 +1}{2} \| \phi^{n+1} - \phi^n \|_2^2
 + \frac{1}{4\dt} \| \phi^{n+1} - \phi^n \|_{-1,N}^2 .
 \label{energy stab-BDF2-0}
\end{equation}
In comparison with the modified energy functional~\eqref{energy stab-0} for the Crank--Nicolson scheme, an $O (\dt)$ deviation away from the original functional is observed in~\eqref{energy stab-BDF2-0} (due to the correction term $\frac{1}{4\dt} \| \phi^{n+1} - \phi^n \|_{-1,N}^2$), while an $O (\dt^2)$ approximation is preserved in~\eqref{energy stab-0}. Therefore, the energy dissipation property, as stated in Theorem \ref{thm:energy stab}, is a closer approximation to the original physical system than the BDF2 approach.
\end{remark}

\section{Numerical experiments}
\label{sect_numerical}

In this section, we will carry out some numerical experiments to verify the theoretical results of the numerical scheme \eqref{scheme-2nd-1} in the two-dimensional case. The choice of the kernel function $J$ in the nonlocal diffusion operator is crucial.
We consider a family of Gauss-type functions
\begin{equation}
\label{kernel}
J_\delta(\bx) = \frac{4}{\pi\delta^4} \mathrm{e}^{-\frac{|\bx|^2}{\delta^2}}, \quad \bx \in \mathbb{R}^2,
\end{equation}
where $\delta>0$ is a parameter.
Obviously, $J_\delta$ defined by \eqref{kernel} is even but not periodic. Note that $J_\delta$ decays to zero exponentially as $|\bx|\to\infty$,
so it is reasonable to view $J_\delta$ as a function supported in $\Omega$
as long as $\delta$ is smaller than the size of $\Omega$.
Then, we can extend it periodically to the whole space to obtain the periodic kernel function.
Since $J_\delta * 1 = 4/\delta^2$, the condition (d) is equivalent to $\delta<2\varepsilon$.
The action of the discrete nonlocal operator $\mathcal{L}_N$ can be implemented by the fast Fourier transform,
and we refer the readers to Lemma 3 in~\cite{du18} for the detailed discussions.

Theoretically, the stabilization constants $A_0$ and $A_1$ should satisfy the restriction \eqref{condition-A0-A1} for the sake of the energy stability.
In practice, we find that the numerical solutions are always located in an interval slightly larger than $[-1,1]$,
and it suffices to set $A_0=2$ and $A_1=5$  for the stability in all the numerical experiments below.
To generate the numerical solution $\phi^1$, we adopt the first-order stabilized semi-implicit scheme (i.e., the scheme (13) studied in~\cite{du18}) with the stabilization constant equal to $2$.

First, we test the temporal convergence rates of the fully-discrete scheme \eqref{scheme-2nd-1}.
We consider the square domain $\Omega=(-1,1)\times(-1,1)$ on which the uniform $1024\times1024$ mesh is adopted. The periodic boundary condition is enforced, and the smooth initial value is taken: 
\[
\phi_0(x,y) = 0.5 \sin \pi x \sin \pi y + 0.1, \quad (x,y) \in \Omega.
\]
The convergence rates will be tested for the cases with various $\varepsilon$ and $\delta$
by computing the numerical solution at time $t=0.05$.
The numerical solutions are computed by the scheme \eqref{scheme-2nd-1} with various time step sizes $\dt=0.005\times 2^{-k}$ with $k=0,1,\dots,8$.
To calculate the numerical errors, we treat the solution computed by $\dt=0.001\times 2^{-8}$ as the benchmark.
Figure \ref{fig_conv} shows the discrete $\ell^2$ errors of the numerical solutions with various $\varepsilon$ and $\delta$. For each case, the second-order temporal convergence rate is obvious.

\begin{figure}[h]
\centering
\includegraphics[width=0.45\textwidth]{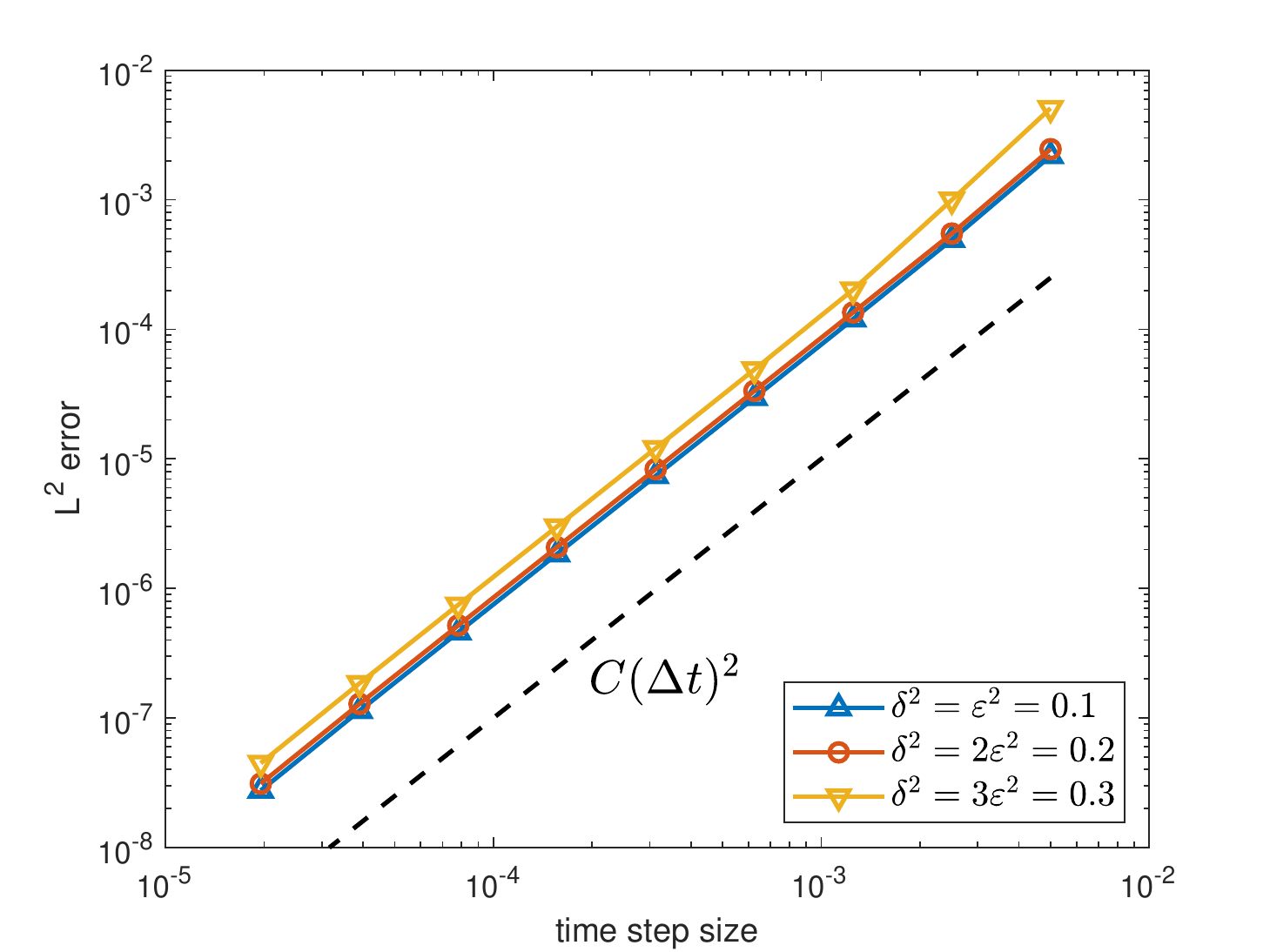}
\includegraphics[width=0.45\textwidth]{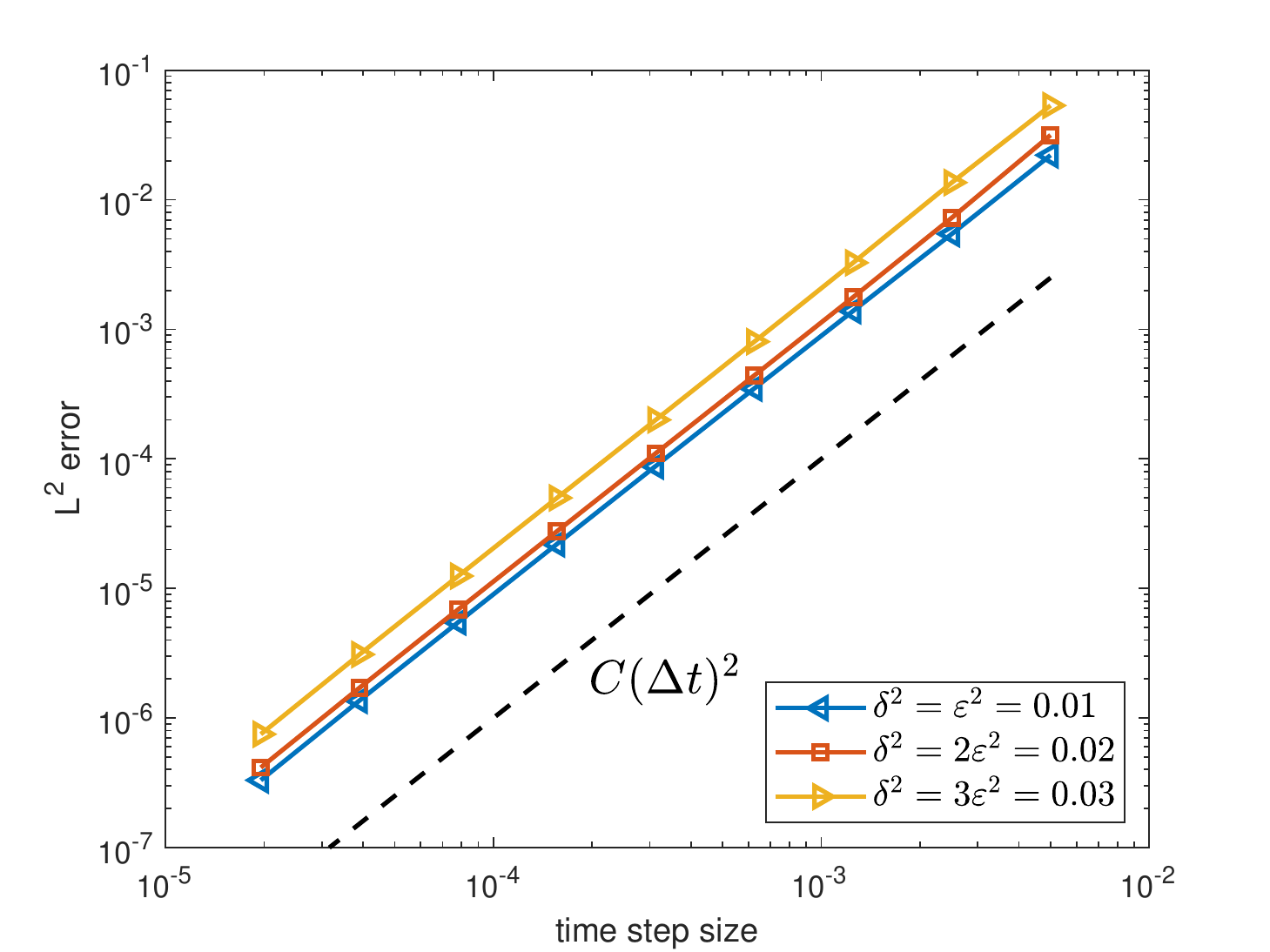}
\caption{Temporal convergence tests: $\varepsilon^2=0.1$ (left) and $\varepsilon^2=0.01$ (right)}
\label{fig_conv}
\end{figure}

Second, we verify the energy stability by simulating the coarsening dynamics. A $t^{-\frac{1}{3}}$ power law of the rate of the energy decay has been predicted in~\cite{DaiDu16}, i.e., $E(t)\sim t^{-\frac13}$, for the classic Cahn--Hilliard equation. Although there has been no similar theoretical analysis for the nonlocal version, we can conduct a numerical simulation of the power law for the NCH equation. Let $\Omega=(-2\pi,2\pi)\times(-2\pi,2\pi)$, and the initial configuration is given by the random data uniformly distributed from $-0.1$ to $0.1$ on each point in a uniform mesh. To accelerate the computations, we adopt variable time step sizes, i.e., $\dt=0.001$ on the time interval $[0,1000)$, $\dt=0.01$ on $[1000,10000)$, and $\dt=0.1$ for $t\ge10000$ (if needed).

With $\delta=0.05$, we perform the simulation on the $512\times512$ spatial mesh.
The evolutions of the energies for the cases $\varepsilon=0.1$, $0.08$, $0.06$, and $0.04$ are displayed in Figure \ref{fig_energy} (left).
For each case, the energy decay is obvious, and the energy decay rate satisfies the $t^{-\frac13}$ power law. More precisely, we can take a logarithmic fitting of the energy in the form $E(t)\sim b_e t^{m_e}$,
namely, a linear fitting applied to $\ln E(t)$ in terms of $\ln t$.
The digits of the coefficients $m_e$ and $b_e$ are collected in Table \ref{table1},
where the values of $m_e$ are close to $-\frac13$, especially when $\varepsilon$ is small.

In addition, we also carry out the simulation with $\delta=0.005$ on the $1024\times1024$ spatial mesh. For the cases $\varepsilon=0.1$, $0.08$, $0.06$, and $0.04$, the right graph in Figure \ref{fig_energy} plots the energy curves and the coefficients of the logarithmic fitting of the energies are listed in Table \ref{table2}, where the $t^{-\frac13}$ power law of the energy decay can be observed. Figure \ref{fig_coarsening} is devoted to the snapshots of the computed solutions
at $t=1$, $10$, $60$, $400$, $2000$, and $10000$ for the coarsening dynamics with $\varepsilon=0.04$. This figure implies the phase transition beginning with the a disorder state towards the order states and the steady state at around $t=10000$.

It is observed that there is no significant difference between these numerical results
and those shown in \cite{du18} and \cite{LiX21b},
although an extra stabilization term $A_1\dt(\phi^{n+1}-\phi^n)$ is used
in comparison with the second-order scheme in \cite{du18}.

\begin{figure}[h]
\centering
\includegraphics[width=0.45\textwidth]{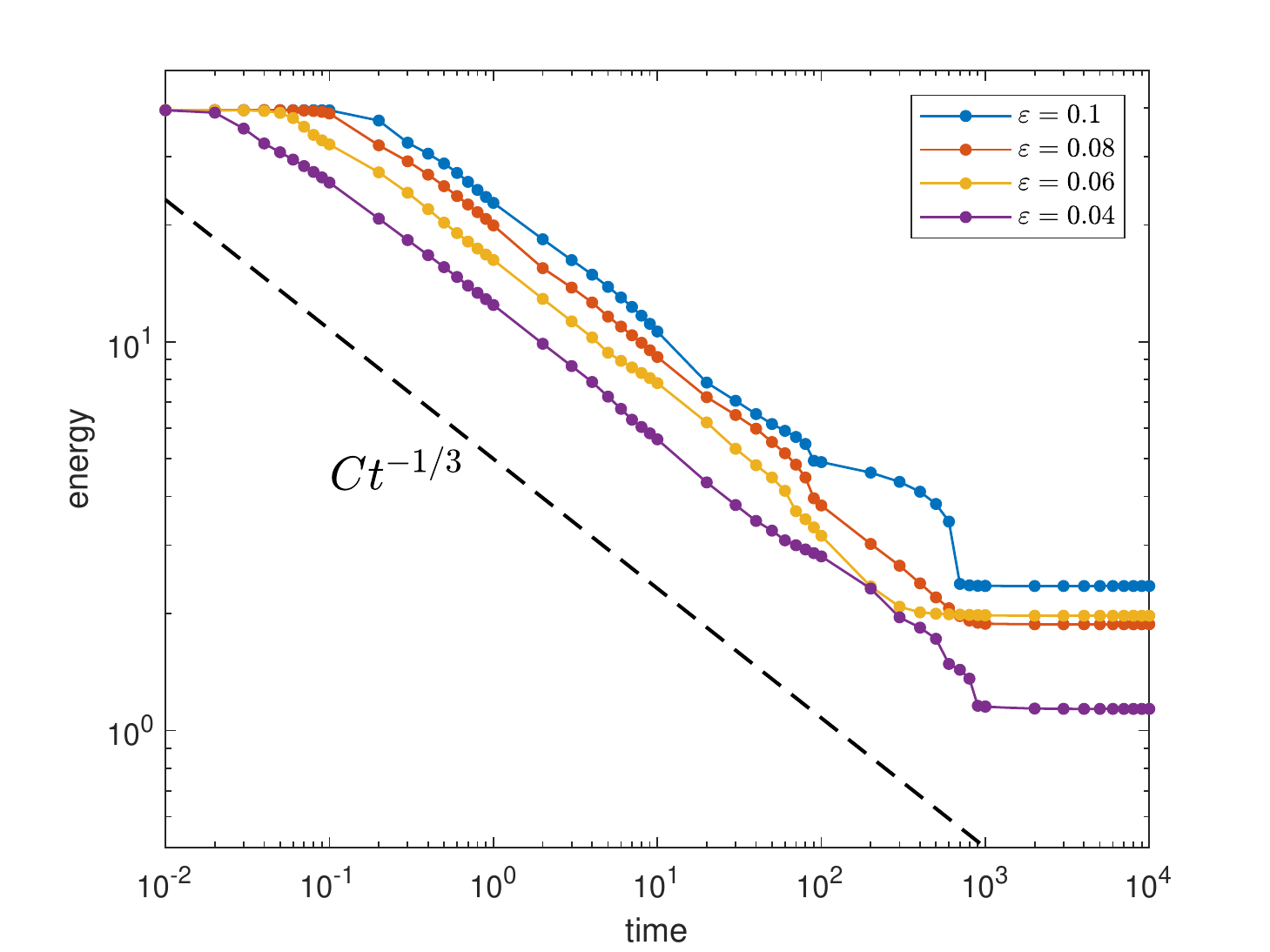}
\includegraphics[width=0.45\textwidth]{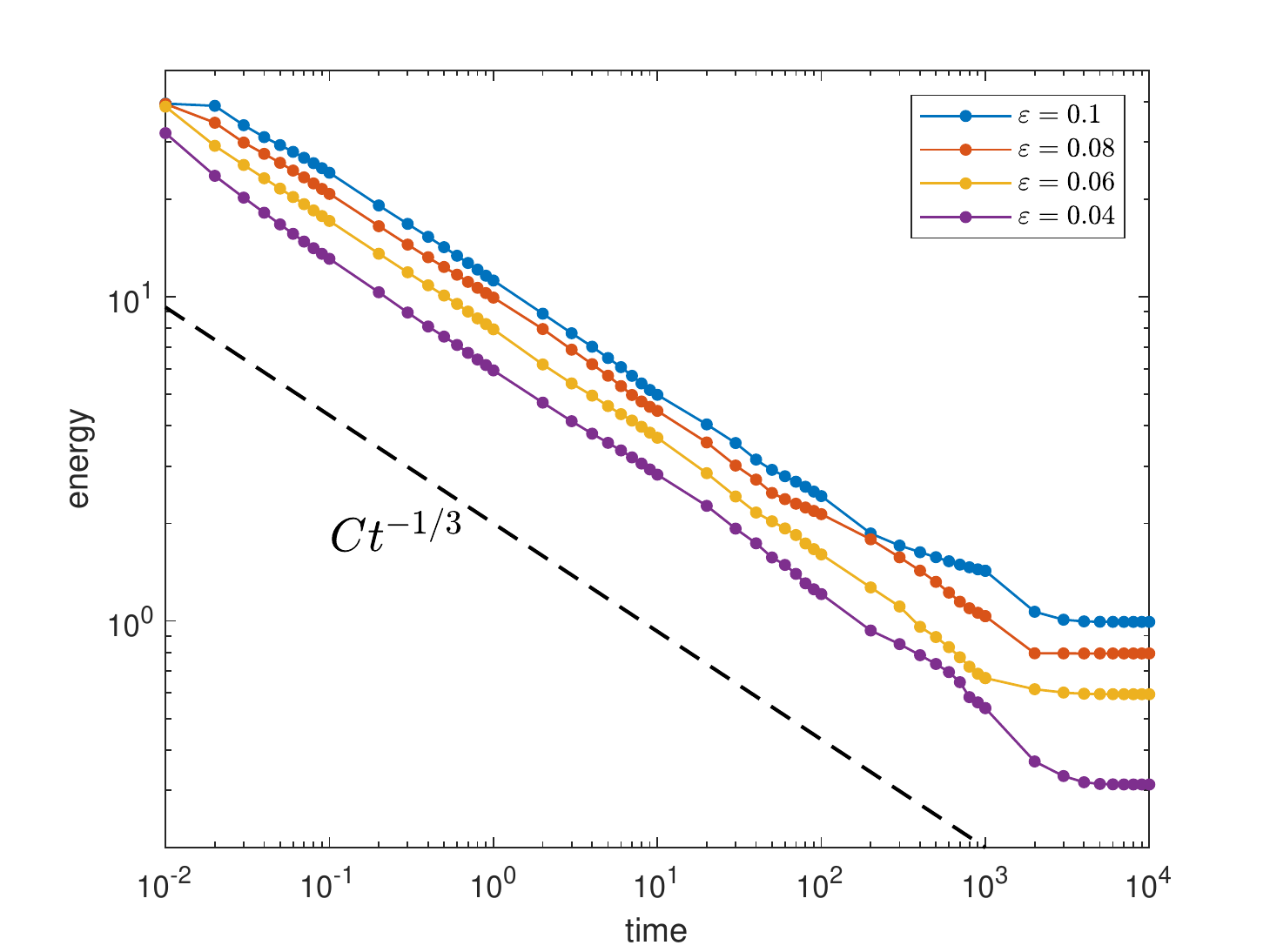}
\caption{Evolutions of the energies for the cases $\delta=0.05$ (left) and $\delta=0.005$ (right)}
\label{fig_energy}
\end{figure}

\begin{table}[h!]\footnotesize\tabcolsep 16pt
\begin{center}
\caption{Coefficients of the fitting $E(t)\sim b_e t^{m_e}$ for the case $\delta=0.05$}\label{table1}\vspace{-2mm}
\end{center}
\begin{center}
\begin{tabular}{c|ccccccc}\toprule
$\varepsilon$ & $0.1$ & $0.09$ & $0.08$ & $0.07$ & $0.06$ & $0.05$ & $0.04$ \\
\hline
$m_e$ & $-0.304$ & $-0.304$ & $-0.323$ & $-0.322$ & $-0.324$ & $-0.333$ & $-0.339$ \\
$b_e$ & $22.447$ & $21.304$ & $19.629$ & $18.090$ & $16.204$ & $14.201$ & $12.324$ \\
\bottomrule
\end{tabular}
\end{center}
\end{table}

\begin{table}[h!]\footnotesize\tabcolsep 16pt
\begin{center}
\caption{Coefficients of the fitting $E(t)\sim b_e t^{m_e}$ for the case $\delta=0.005$}\label{table2}\vspace{-2mm}
\end{center}
\begin{center}
\begin{tabular}{c|ccccccc}\toprule
$\varepsilon$ & $0.1$ & $0.09$ & $0.08$ & $0.07$ & $0.06$ & $0.05$ & $0.04$ \\
\hline
$m_e$ & $-0.343$ & $-0.331$ & $-0.337$ & $-0.335$ & $-0.336$ & $-0.349$ & $-0.330$ \\
$b_e$ & $11.158$ & $10.534$ & $9.769$ & $8.854$ & $7.940$ & $6.964$ & $6.009$ \\
\bottomrule
\end{tabular}
\end{center}
\end{table}

\begin{figure}[!htp]
\centering
\includegraphics[width=0.33\textwidth]{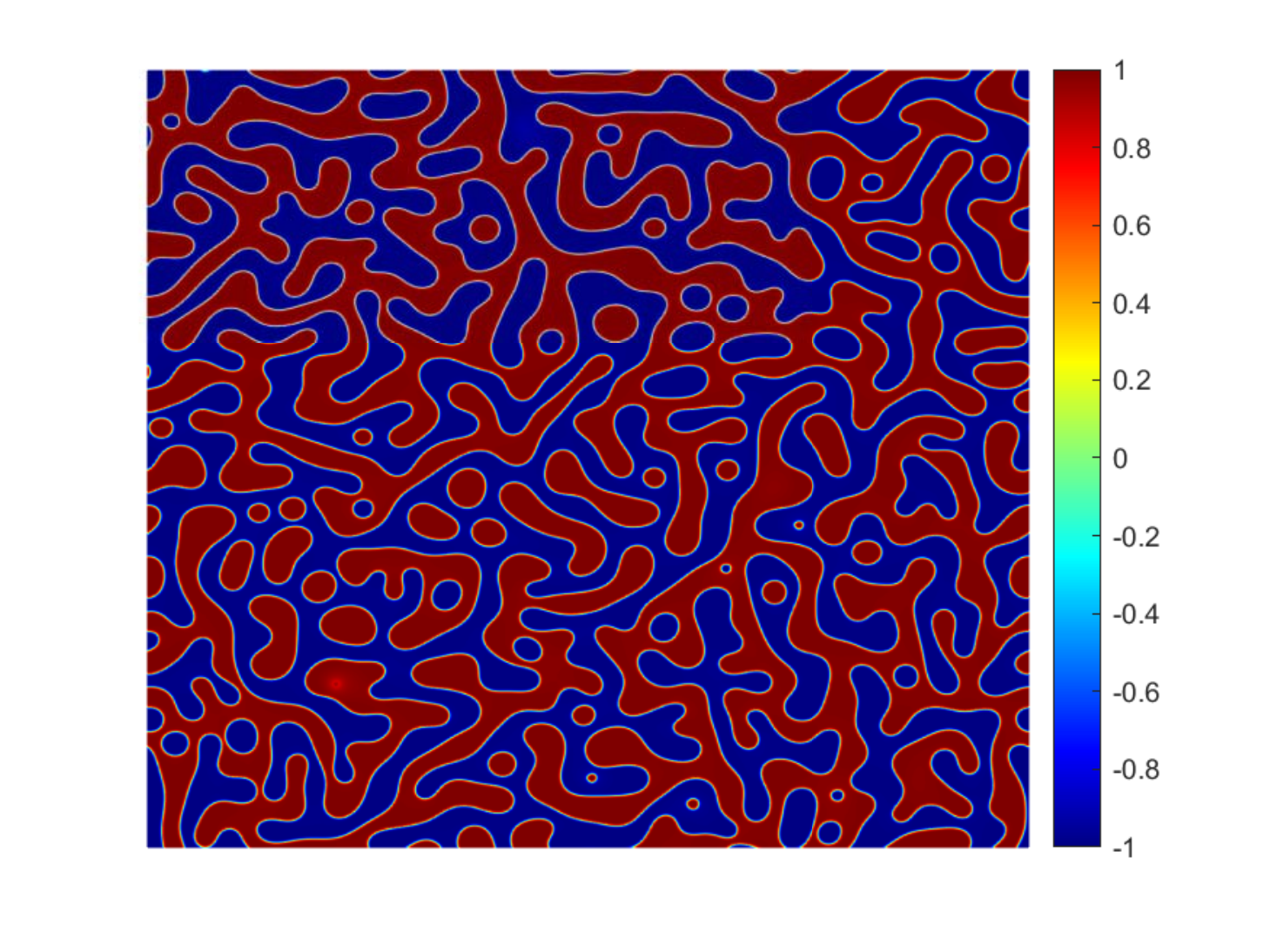}\hspace{-0.4cm}
\includegraphics[width=0.33\textwidth]{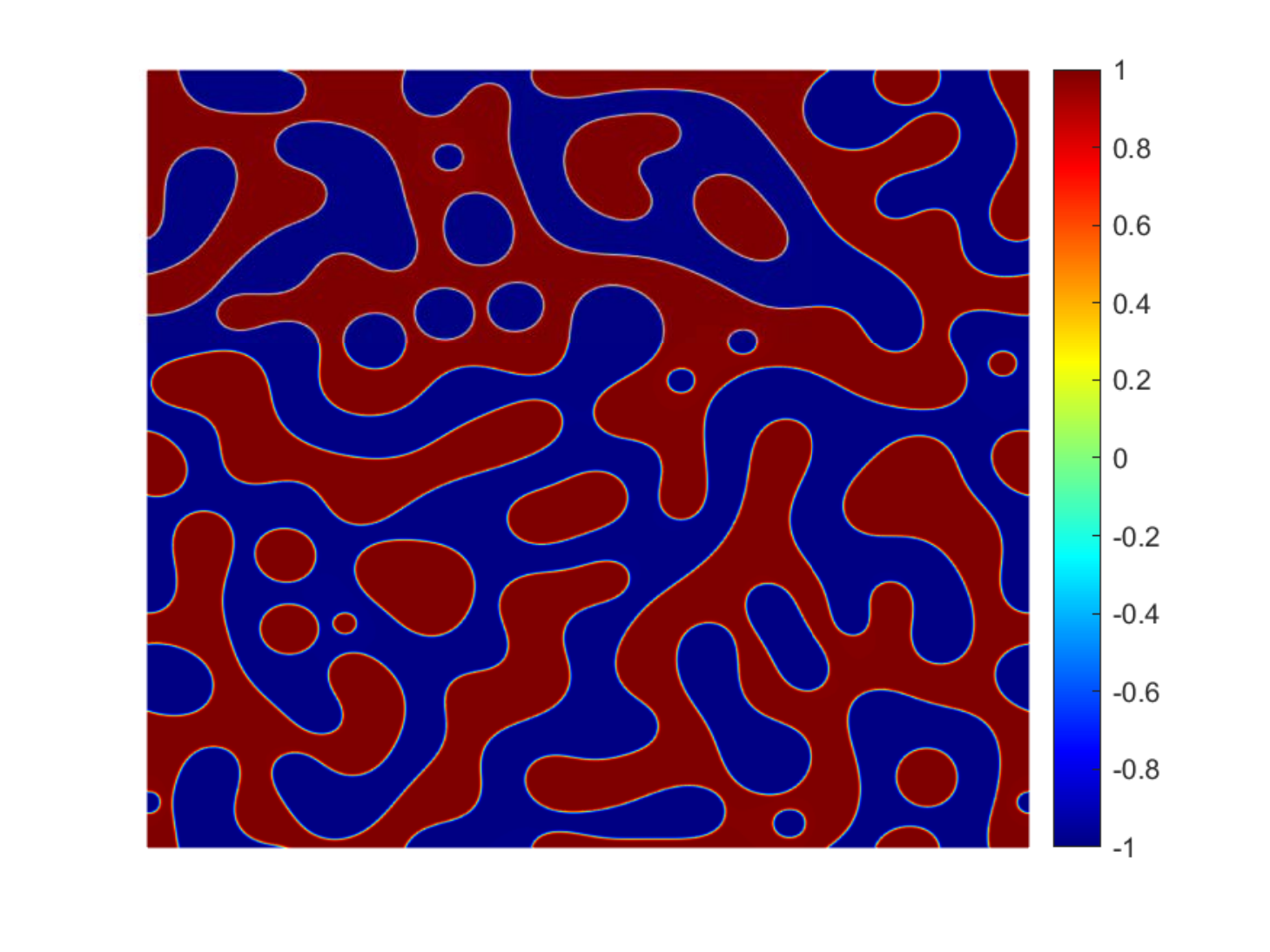}\hspace{-0.4cm}
\includegraphics[width=0.33\textwidth]{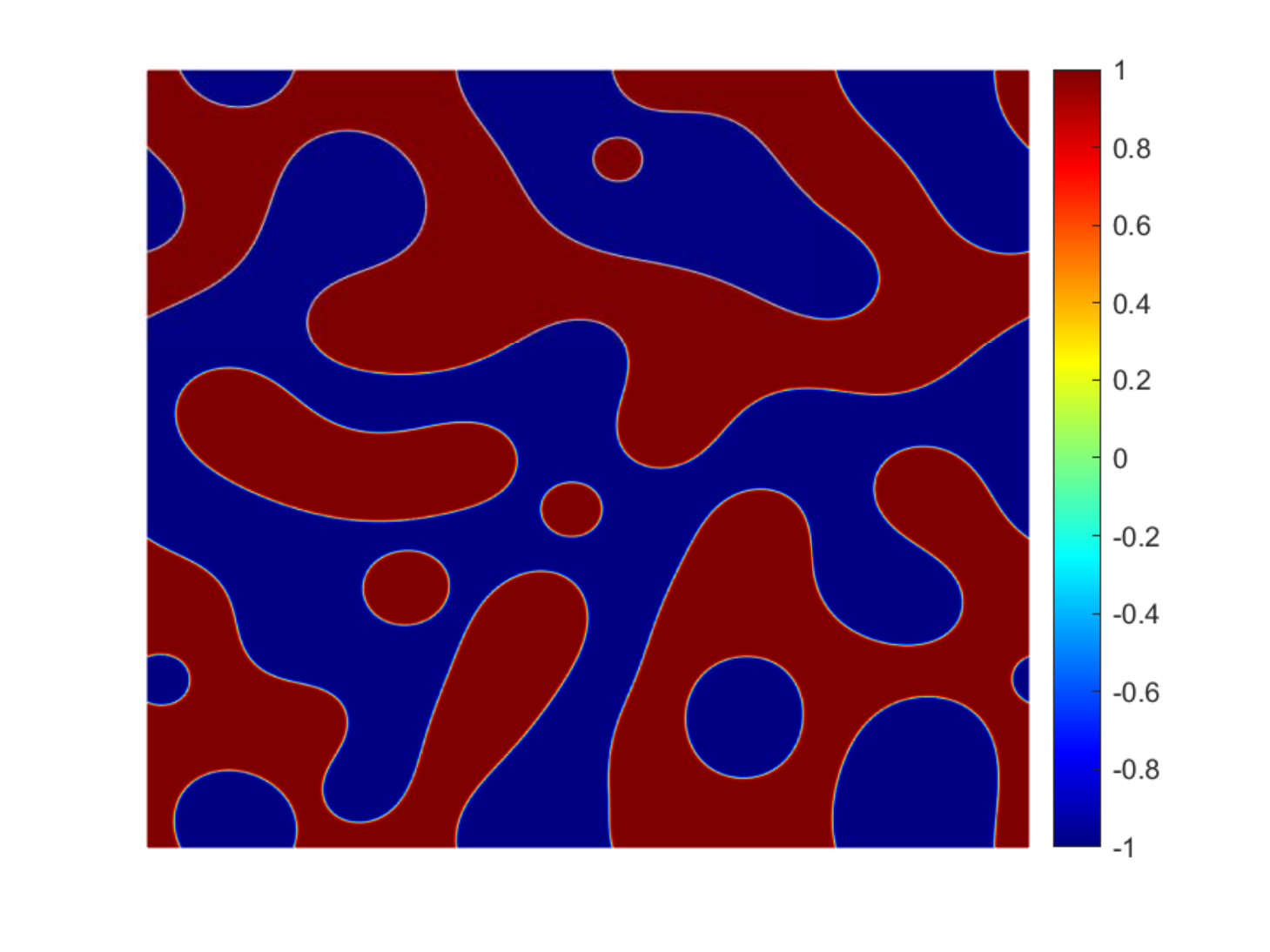}\vspace{-0.2cm}
\includegraphics[width=0.33\textwidth]{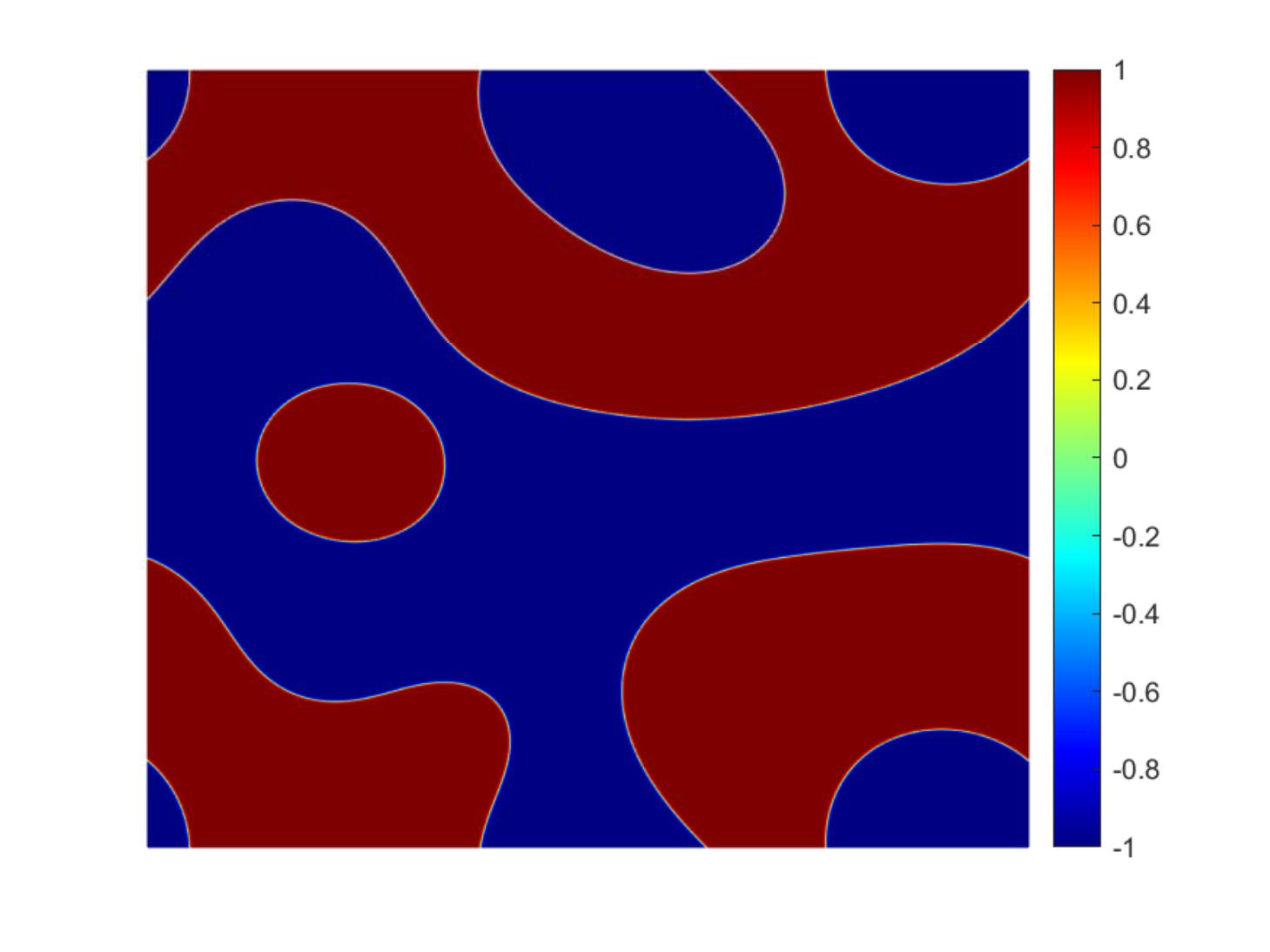}\hspace{-0.4cm}
\includegraphics[width=0.33\textwidth]{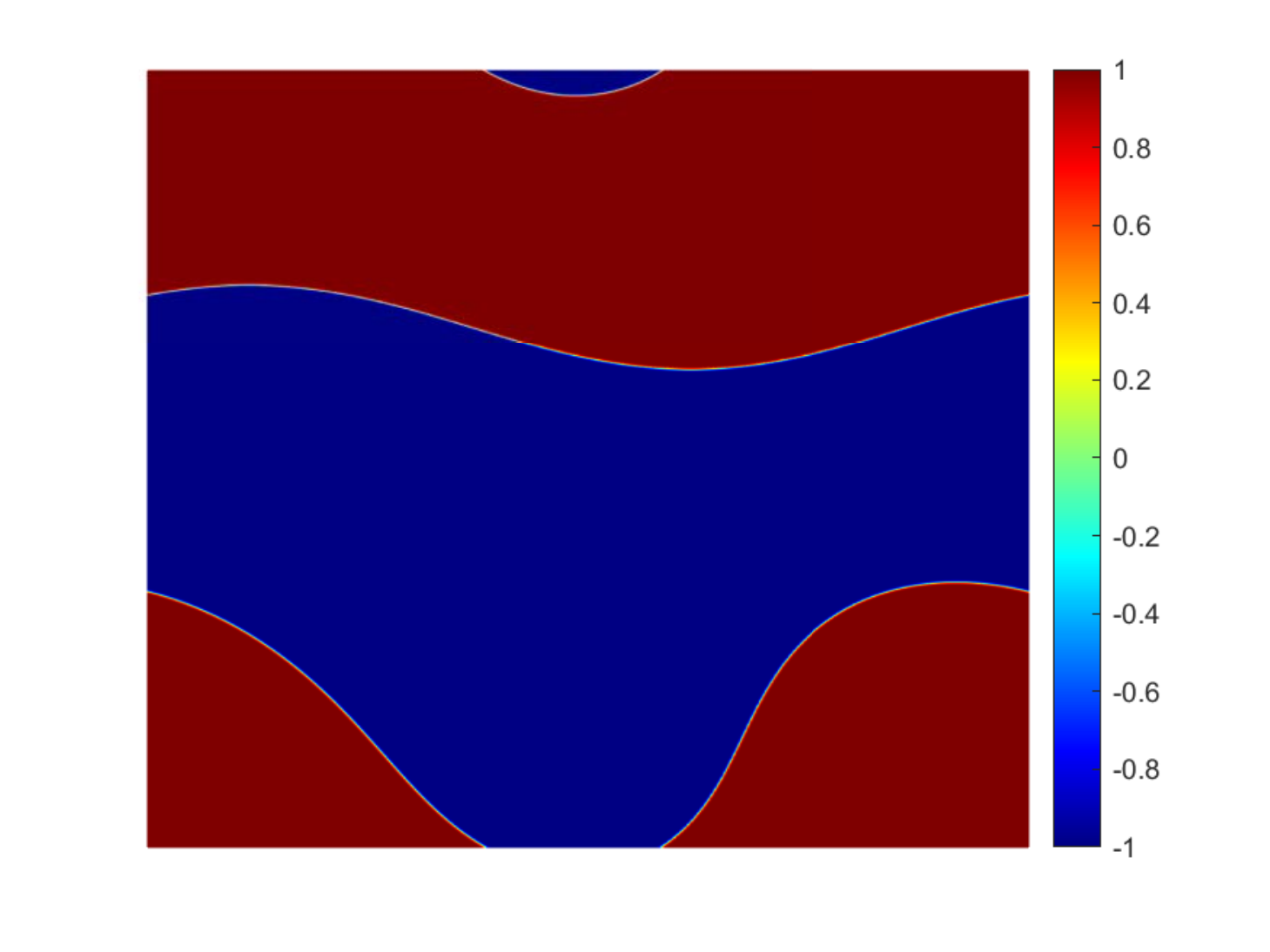}\hspace{-0.4cm}
\includegraphics[width=0.33\textwidth]{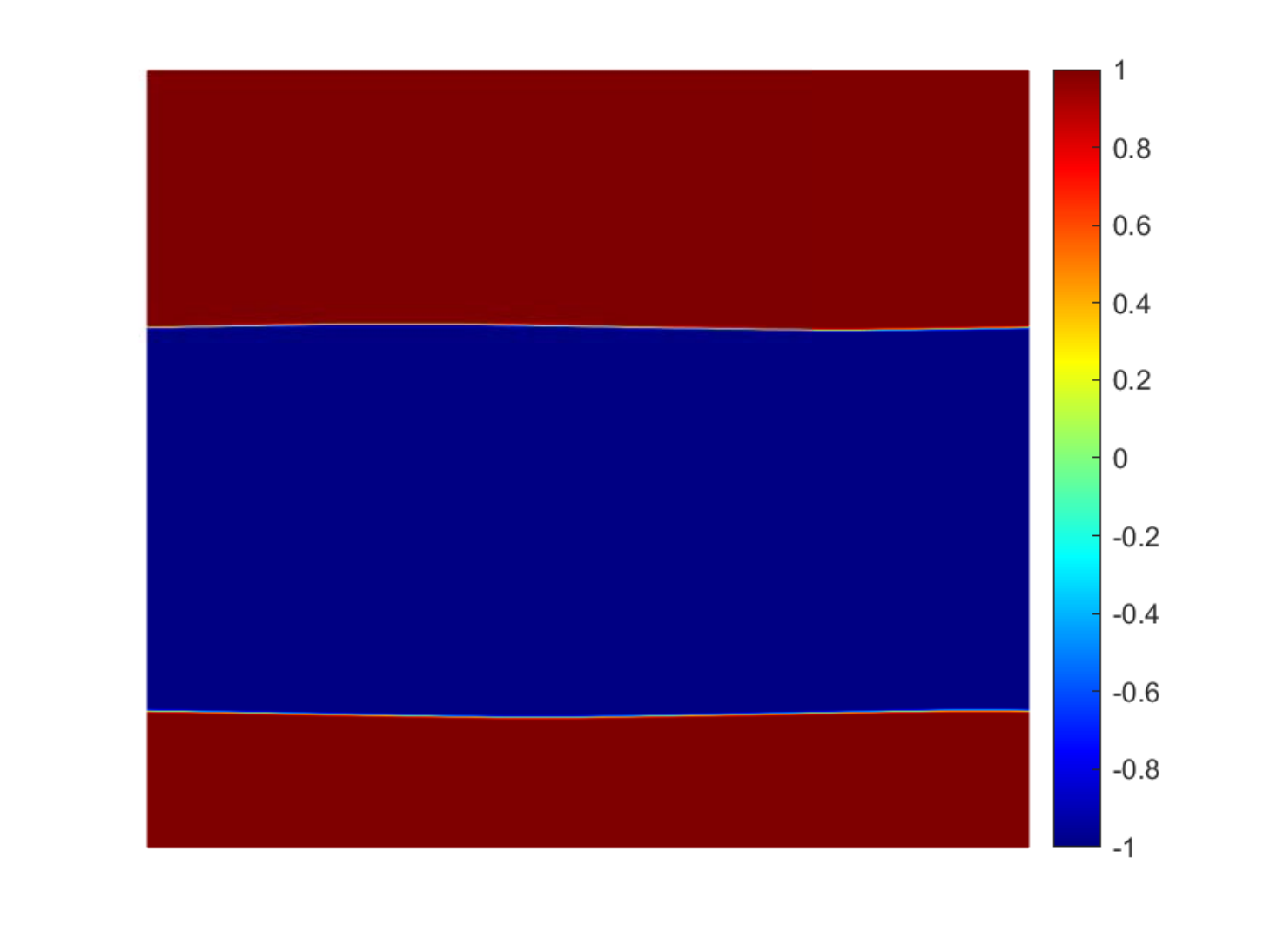}\vspace{-0.2cm}
\caption{Snapshots of the coarsening dynamics at $t=1$, $10$, $60$, $400$, $2000$, and $10000$
for the case $\delta=0.005$ and $\varepsilon=0.04$}
\label{fig_coarsening}
\end{figure}

\section{Conclusion}
\label{sect_conclusion}
	
In this work, we study a second-order stabilized linear numerical scheme for the nonlocal Cahn--Hilliard equation. A modified Crank--Nicolson and second-order explicit extrapolation are adopted for the temporal discretization. To ensure the energy stability at a theoretical level, we add two artificial stabilization terms, $A_0 \Delta_N ( \phi^{n+1} - 2 \phi^n + \phi^{n-1})$ and $A_1 \dt \Delta_N ( \phi^{n+1} - \phi^n )$, in the numerical scheme.
In particular, the optimal rate convergence analysis is accomplished by applying the higher-order consistency estimate, combined with a rough error estimate and a refined error estimate. In turn, the $\ell^\infty$ bound of the numerical solution, as well as its discrete temporal derivative, becomes an important by-product. Meanwhile, the energy stability is obtained in the sense that a modified energy decreases in time and the original energy is uniformly bounded, where the second stabilization term has played an important role. The theoretical result has greatly improved the ones reported in an existing work~\cite{du18}, in which the second-order scheme can be viewed as a special case of the proposed scheme \eqref{scheme-2nd-1} with $A_1=0$. In comparison with the second-order scheme based on the BDF2 temporal discretization in \cite{LiX21b}, the lower bounds required for $A_0$ and $A_1$ in \eqref{condition-A0-A1} are moderately smaller, which implies that the constraint for the energy stability is less restrictive than that for the BDF2 scheme. Moreover, the modified energy defined by \eqref{energy stab-0} gives an approximation of the original energy with a deviation of order $O(\dt^2)$, while an $O(\dt)$ correction term is added for the modification adopted in the BDF2 scheme~\cite{LiX21b}. In other words, the energy dissipation property (Theorem \ref{thm:energy stab}) turns out to be closer to the original physical system than the BDF2 approach.


\Acknowledgements{
This work was supported by the CAS AMSS-PolyU Joint Laboratory of Applied Mathematics.
The first author was supported by the Hong Kong Research Council General Research Fund (Grant No. 15300821)
and the Hong Kong Polytechnic University grants (Grant Nos. 1-BD8N, 4-ZZMK, and 1-ZVWW).
The second author was supported by the Hong Kong Research Council Research Fellow Scheme (Grant No. RFS2021-5S03)
and General Research Fund (Grant No. 15302919).
The third author was supported by US National Science Foundation (Grant No. DMS-2012269).}





\begin{thebibliography}{99}

\bibitem{AinsworthMao17}
M.~Ainsworth and Z.~Mao.
Analysis and approximation of a fractional Cahn--Hilliard equation.
{\em SIAM J. Numer. Anal.}, 55:1689--1718, 2017.

\bibitem{archer04a}
A.~Archer and R.~Evans.
Dynamical density functional theory and its application to spinodal
  decomposition.
 {\em J. Chem. Phys.}, 121:4246--4254, 2004.

\bibitem{archer04b}
A.~Archer and M.~Rauscher.
 Dynamical density functional theory for interacting {Brownian}
  particles: Stochastic or deterministic?
 {\em J. Phys. A: Math. Gen.}, 37:9325, 2004.

\bibitem{baskaran13b}
A. Baskaran, J. S. Lowengrub, C. Wang, and S.M. Wise.
 Convergence analysis of a second order convex splitting scheme for the modified phase field crystal equation.
 {\em SIAM J. Numer. Anal.}, 51:2851--2873, 2013.

\bibitem{bates06a}
P.~Bates.
 On some nonlocal evolution equations arising in materials science.
 In Hermann Brunner, Xiao-Qiang Zhao, and Xingfu Zou, editors, {\em
  Nonlinear Dynamics and Evolution Equations}, volume~48 of {\em Fields
  Institute Communications}, pages 13--52. American Mathematical Society,
  Providence, RI; USA, 2006.

\bibitem{bates09}
P.~Bates, S.~Brown, and J.~Han.
 Numerical analysis for a nonlocal {Allen-Cahn} equation.
 {\em Int. J. Numer. Anal. Model.}, 6:33--49, 2009.

\bibitem{bates05b}
P.~Bates and J.~Han.
 The {Dirichlet} boundary problem for a nonlocal {Cahn-Hilliard}
  equation.
 {\em J. Math. Anal. Appl.}, 311:289, 2005.

\bibitem{bates05a}
P.~Bates and J.~Han.
 The {Neumann} boundary problem for a nonlocal {Cahn-Hilliard}
  equation.
 {\em J. Diff. Eqs.}, 212:235--277, 2005.

\bibitem{bates06b}
P.~Bates, J.~Han, and G.~Zhao.
 On a nonlocal phase-field system.
 {\em Nonlinear Analysis: Theory, Methods and Applications},
  64:2251--2278, 2006.

\bibitem{cahn58}
J.~Cahn and J.~Hilliard.
 Free energy of a nonuniform system. I. Interfacial free energy.
 {\em J. Chem. Phys.}, 28:258, 1958.

\bibitem{cheng2019a}
K. Cheng, W. Feng, C. Wang, and S.M. Wise.
  An energy stable fourth order finite difference scheme for the Cahn-Hilliard equation.
 {\em J. Comput. Appl. Math.}, 362:574--595, 2019.

\bibitem{cheng16a}
K. Cheng, C. Wang, S.M. Wise, and X. Yue.
 A second-order, weakly energy-stable pseudo-spectral
scheme for the Cahn-Hilliard equation and its solution by the homogeneous linear iteration method.
 {\em J. Sci. Comput.}, 69:1083--1114, 2016.

\bibitem{DaiDu16}
S.~Dai and Q.~Du.
Computational studies of coarsening rates for the Cahn-Hilliard equation with phase-dependent diffusion mobility.
{\em J. Comput. Phys.}, 310:85--108, 2016.

\bibitem{diegel17}
A. Diegel, C. Wang, X. Wang, and S.M. Wise.
  Convergence analysis and error estimates for a second order accurate finite element method for the Cahn-Hilliard-Navier-Stokes system.
 {\em Numer. Math.}, 137:495--534, 2017.

\bibitem{diegel16}
A. Diegel, C. Wang, and S.M. Wise.
  Stability and convergence of a second order mixed  finite element method for the Cahn-Hilliard  equation.
 {\em IMA J. Numer. Anal.}, 36:1867--1897, 2016.

\bibitem{du12a}
Q.~Du, M.~Gunzburger, R.~Lehoucq, and K.~Zhou.
 Analysis and approximation of nonlocal diffusion problems with volume
  constraints.
 {\em SIAM Rev.}, 54:667--696, 2012.

\bibitem{du18}
Q.~Du, L.~Ju, X.~Li, and Z.~Qiao.
 Stabilized linear semi-implicit schemes for the nonlocal
  {Cahn-Hilliard} equation.
 {\em J. Comput. Phys.}, 363:39--54, 2018.

\bibitem{du19}
Q.~Du, L.~Ju, X.~Li, and Z.~Qiao.
 Maximum principle preserving exponential time differencing schemes
  for the nonlocal {Allen-Cahn} equation.
 {\em SIAM J. Numer. Anal.}, 57:876--898, 2019.

\bibitem{du91}
Q.~Du and R.~Nicolaides.
 Numerical analysis of a continuum model of a phase transition.
 {\em SIAM J. Numer. Anal.}, 28:1310--1322, 1991.

\bibitem{du16}
Q.~Du and J.~Yang.
 Asymptotically compatible {Fourier} spectral approximations of
  nonlocal {Allen-Cahn} equations.
 {\em SIAM J. Numer. Anal.}, 54:1899--1919, 2016.

\bibitem{duan20a}
C. Duan, C. Liu, C. Wang, and X. Yue.
 Convergence analysis of a numerical scheme for the porous medium equation by an energetic variational approach.
 {\em Numer. Math. Theor. Meth. Appl.}, 13:1--18, 2020.

\bibitem{duan21a}
C. Duan, W. Chen, C. Liu, C. Wang, and S. Zhou.
Convergence analysis of structure-preserving numerical methods for nonlinear Fokker–Planck equations with nonlocal interactions.
 {\em Math. Meth. App. Sci.}, 45:3764--3781, 2022.

\bibitem{E95}
W. E and J.-G. Liu.
Projection method I: Convergence and numerical boundary layers.
 {\em SIAM J. Numer. Anal.}, 32:1017--1057, 1995.

\bibitem{eyre98}
D.~Eyre.
 Unconditionally gradient stable time marching the {C}ahn-{H}illiard equation.
 In J.~W. Bullard, R.~Kalia, M.~Stoneham, and L.Q. Chen, editors, {\em
  Computational and Mathematical Models of Microstructural Evolution},
  volume~53, pages 1686--1712, Warrendale, PA, USA, 1998. Materials Research
  Society.

\bibitem{fife03}
P.C. Fife.
 Some nonclassical trends in parabolic and parabolic-like evolutions.
 In M.~Kirkilionis, S.~Kromker, R.~Rannacher, and F.~Tomi, editors,
  {\em Trends in Nonlinear Analysis}, Chapter~3, pages 153--191. Springer,
  2003.

\bibitem{gottlieb12a}
S.~Gottlieb, F.~Tone, C.~Wang, X.~Wang, and D.~Wirosoetisno.
 Long time stability of a classical efficient scheme for two
  dimensional {Navier-Stokes} equations.
 {\em SIAM J. Numer. Anal.}, 50:126--150, 2012.

\bibitem{gottlieb12b}
S.~Gottlieb and C.~Wang.
 Stability and convergence analysis of fully discrete {Fourier}
  collocation spectral method for {3-D} viscous {Burgers'} equation.
 {\em J. Sci. Comput.}, 53:102--128, 2012.

\bibitem{guan17a}
Z.~Guan, J.S. Lowengrub, and C.~Wang.
 Convergence analysis for second order accurate schemes for the
  periodic nonlocal {Allen-Cahn} and {Cahn-Hilliard} equations.
 {\em Math. Methods Appl. Sci.}, 40(18):6836--6863, 2017.

\bibitem{guan14b}
Z.~Guan, J.S. Lowengrub, C.~Wang, and S.M. Wise.
 Second-order convex splitting schemes for nonlocal {Cahn-Hilliard}
  and {Allen-Cahn} equations.
 {\em J. Comput. Phys.}, 277:48--71, 2014.

\bibitem{guan14a}
Z.~Guan, C.~Wang, and S.M. Wise.
 A convergent convex splitting scheme for the periodic nonlocal
  {Cahn-Hilliard} equation.
 {\em Numer. Math.}, 128:377--406, 2014.

\bibitem{guo16}
J. Guo, C. Wang, S.M. Wise, and X. Yue.
An $H^2$ convergence of a second-order convex-splitting, finite difference scheme for the three-dimensional Cahn-Hilliard equation.
 {\em Commun. Math. Sci.}, 14:489--515, 2016.

\bibitem{guo2021}
J. Guo, C. Wang, S.M. Wise, and X. Yue.
An improved error analysis for a second-order numerical scheme for the Cahn-Hilliard equation.
 {\em J. Comput. Appl. Math.}, 388:113300, 2021.

\bibitem{hornthrop01}
D.~Hornthrop, M.~Katsoulakis, and D.~Vlachos.
 Spectral methods for mesoscopic models of pattern formation.
 {\em J. Comput. Phys.}, 173:364--390, 2001.

\bibitem{JuLiQi22}
L.~Ju, X.~Li, and Z.~Qiao.
Generalized SAV-exponential integrator schemes for Allen--Cahn type gradient flows.
 {\em SIAM J. Numer. Anal.}, 60:1905--1931, 2022.

\bibitem{LiD2017}
D.~Li and Z.~Qiao.
 On second order semi-implicit {Fourier spectra}l methods for {2D
  Cahn-Hilliard} equations.
 {\em J. Sci. Comput.}, 70:301--341, 2017.

\bibitem{LiD2017b}
D.~Li and Z.~Qiao.
 On the stabilization size of semi-implicit {Fourier-spectral} methods
  for {3D Cahn-Hilliard} equations.
 {\em Commun. Math. Sci.}, 15:1489--1506, 2017.

\bibitem{LiD2016a}
D.~Li, Z.~Qiao, and T.~Tang.
 Characterizing the stabilization size for semi-implicit
  {Fourier-spectral} method to phase field equations.
 {\em SIAM J. Numer. Anal.}, 54:1653--1681, 2016.

\bibitem{LiX21a}
X.~Li, Z.~Qiao, and C.~Wang.
 Convergence analysis for a stabilized linear semi-implicit numerical
  scheme for the nonlocal {Cahn-Hilliard} equation.
 {\em Math. Comp.}, 90:171--188, 2021.

\bibitem{LiX21b}
X.~Li, Z.~Qiao, and C.~Wang.
 Stabilization parameter analysis of a second order linear numerical scheme for the nonlocal Cahn-Hilliard equation.
 {\em IMA J. Numer. Anal.}, 2022, https://doi.org/10.1093/imanum/drab109.

\bibitem{LiQiZh16}
X.~Li, Z.~Qiao, and H.~Zhang.
An unconditionally energy stable finite difference scheme for a stochastic Cahn-Hilliard equation.
 {\em Sci. China Math.}, 59:1815--1834, 2016.

\bibitem{LiShen22}
X.~Li and J.~Shen.
Efficient linear and unconditionally energy stable schemes for the modified phase field crystal equation.
 {\em Sci. China Math.}, 65:2201--2218, 2022.

\bibitem{LiaoSoTaZh21}
H.~Liao, X.~Song, T.~Tang, and T.~Zhou.
Analysis of the second-order BDF scheme with variable steps for the molecular beam epitaxial model without slope selection.
{\em Sci. China Math.}, 64:887--902, 2021.

\bibitem{LiaoZh20}
H.~Liao and Z.~Zhang.
Analysis of adaptive BDF2 scheme for diffusion equations.
{\em Math. Comp.}, 90:1207--1226, 2021.

\bibitem{LiuC2021}
C. Liu, C. Wang, S.M. Wise, X. Yue, and S. Zhou.
A positivity-preserving, energy stable and convergent numerical scheme for the Poisson-Nernst-Planck system.
 {\em Math. Comp.}, 90:2071--2106, 2021.

\bibitem{mclachlan99}
R.I. McLachlan, G.R.W. Quispel, and N.~Robidoux.
 Geometric integration using discrete gradients.
 {\em R. Soc. Lond. Philos. Trans. Ser. A Math. Phys. Eng. Sci.},
  357:1021–1045, 1999.

\bibitem{Meng20}
 X. Meng, Z. Qiao, C. Wang, and Z. Zhang.
 Artificial regularization parameter analysis for the no-slope-selection epitaxial thin film model.
 {\em CSIAM Trans. Appl. Math.}, 1:441--462, 2020.

\bibitem{QiaoZhTa11}
Z.~Qiao, Z.~Zhang, and T.~Tang.
An adaptive time-stepping strategy for the molecular beam epitaxy models.
{\em SIAM J. Sci. Comput.}, 33:1395--1414, 2011.

\bibitem{STWW03}
 R. Samelson, R. Temam, C. Wang, and S. Wang.
  Surface pressure Poisson equation formulation of the primitive equations: Numerical schemes.
 {\em SIAM J. Numer. Anal.}, 41:1163--1194, 2003.

\bibitem{shen12}
J.~Shen, C.~Wang, X.~Wang, and S.M. Wise.
 Second-order convex splitting schemes for gradient flows with
  {Ehrlich-Schwoebel} type energy: Application to thin film epitaxy.
 {\em SIAM J. Numer. Anal.}, 50:105--125, 2012.

\bibitem{ShenXuYa19}
J. Shen, J. Xu, and J. Yang.
A new class of efficient and robust energy stable schemes for gradient flows.
{\em SIAM Rev.}, 61:474--506, 2019.

\bibitem{shen10}
J.~Shen and X.~Yang.
 Numerical approximations of {Allen-Cahn} and {Cahn-Hilliard}
  equations.
 {\em Discrete Contin. Dyn. Syst.}, 28:1669--1691, 2010.

\bibitem{SongXuKa16}
F.~Song, C.~Xu, and G.~E.~Karniadakis.
A fractional phase-field model for two-phase flows with tunable sharpness: Algorithms and simulations.
{\em Comput. Methods Appl. Mech. Engrg.}, 305:376--404, 2016.

\bibitem{TangYuZh19}
T.~Tang, H.~Yu, and T.~Zhou.
On energy dissipation theory and numerical stability for time-fractional phase-field equations.
{\em SIAM J. Sci. Comput.}, 41:A3757--A3778, 2019.

\bibitem{Temam2001}
R. Temam.
 Navier-Stokes Equations: Theory and Numerical Analysis.
 Providence, Rhode Island: American Mathematical Society, 2001.

\bibitem{WLJ04}
C.~Wang, J.-G. Liu, and H. Johnston.
 Analysis of a fourth order finite difference method for incompressible Boussinesq equations.
 {\em Numer. Math.}, 97:555--594, 2004.

\bibitem{wang10}
C.~Wang, X.~Wang, and S.M. Wise.
 Unconditionally stable schemes for equations of thin film epitaxy.
 {\em Discrete Cont. Dyn. Sys. Ser. A}, 28:405--423, 2010.

\bibitem{WangL15}
L. Wang, W. Chen, and C.~Wang.
 An energy-conserving second order numerical scheme for nonlinear hyperbolic equation with an exponential nonlinear term.
 {\em J. Comput. Appl. Math.}, 280:347--366, 2015.

\bibitem{WangL2018}
L. Wang and H. Yu.
  On efficient second order stabilized semi-implicit schemes for the Cahn-Hilliard phase-field equation.
 {\em J. Sci. Comput.}, 77:1185--1209, 2018.

\bibitem{WangL2018b}
L. Wang and H. Yu.
Convergence analysis of an unconditionally energy stable linear Crank-Nicolson scheme for the Cahn-Hilliard equation.
{\em J. Math. Study}, 51:89--114, 2018.

\bibitem{WangL2019}
L. Wang and H. Yu.
Energy-stable second-order linear schemes for the Allen–Cahn phase-field equation.
{\em Commun. Math. Sci.}, 17:609--635, 2019.

\bibitem{WangL2020}
L. Wang and H. Yu.
An energy stable linear diffusive Crank–Nicolson scheme for the Cahn–Hilliard gradient flow.
{\em J. Comput. Appl. Math.}, 377:112880, 2020.

\bibitem{wise09}
S.M. Wise, C.~Wang, and J.S. Lowengrub.
 An energy-stable and convergent finite-difference scheme for the
  phase field crystal equation.
 {\em SIAM J. Numer. Anal.}, 47:2269--2288, 2009.

\bibitem{xu06}
C.~Xu and T.~Tang.
 Stability analysis of large time-stepping methods for epitaxial
  growth models.
 {\em SIAM J. Numer. Anal.}, 44:1759--1779, 2006.

 \bibitem{yan18}
 Y. Yan, W. Chen, C. Wang, and S.M. Wise.
 A second-order energy stable {BDF} numerical scheme for the Cahn-Hilliard equation.
 {\em Commun. Comput. Phys.}, 23:572--602, 2018.

\bibitem{YangZh20}
X. Yang and G. Zhang.
Convergence analysis for the invariant energy quadratization (IEQ) schemes for solving the Cahn--Hilliard and Allen--Cahn equations with general nonlinear potential.
{\em J. Sci. Comput.}, 82:55, 2020.

\end{thebibliography}
\end{document}